\documentclass[reqno]{amsart}

\usepackage{amsmath}
\usepackage{hyperref}
\usepackage[normalem]{ulem}
\hypersetup
{
    colorlinks,
    citecolor=blue,
    filecolor=red,
    linkcolor=blue,
    urlcolor=black
}
\usepackage{amssymb}
\usepackage{amsthm}
\usepackage{extarrows}
\usepackage{enumitem}
\usepackage{mathrsfs}
\usepackage{color}
\usepackage
[
a4paper,
vmargin=3cm,
hmargin=2.7cm
]{geometry}

\newcommand\R{\mathbb R}
\newcommand\Z{\mathbb Z}
\newcommand\T{\mathbb T}
\newcommand\N{\mathbb N}
\newcommand\E{\mathbb E}
\newcommand\e{\varepsilon}
\newcommand\p{\mathbb P}

\newcommand\W{\mathcal W}
\newcommand\s{\mathcal S}

\newcommand\B{\mathcal B}

\newcommand\Pm[2][{\mathscr{P} }]{{\mathscr P}_{#2}}
\let\a=\alpha

\let\g=\gamma

\theoremstyle{plain}
\numberwithin{equation}{section}
\newtheorem{Theorem}{Theorem}[section]
\newtheorem{Proposition}{Proposition}[section]
\newtheorem{Lemma}{Lemma}[section]
\newtheorem{Definition}{Definition}[section]
\theoremstyle{definition}
\newtheorem{Remark}{Remark}[section]

\newtheorem{Assumption}{Assumption}

\begin{document}

\title[On the Stochastic DGH Equation]{On the stochastic Dullin-Gottwald-Holm equation:\\ Global existence and wave-breaking phenomena}

\author[C. Rohde]{Christian Rohde}
\address{Institut f\"{u}r Angewandte Analysis und Numerische Simulation\\ Universit\"{a}t Stuttgart\\  Pfaffenwaldring 57, 70569 Stuttgart, Germany}
\email{christian.rohde@mathematik.uni-stuttgart.de}

\author[H. Tang]{Hao Tang}
\address{Institut f\"{u}r Angewandte Analysis und Numerische Simulation\\ Universit\"{a}t Stuttgart\\  Pfaffenwaldring 57, 70569 Stuttgart, Germany}
\email{Hao.Tang@mathematik.uni-stuttgart.de}
\thanks{C.~Rohde acknowledges  support by Deutsche Forschungsgemeinschaft (DFG,
	German Research Foundation) under Germany's Excellence Strategy - EXC
	2075 - 39074001. H.~Tang is supported by a postdoctoral scholarship of the Alexander von Humboldt Foundation.}


\date{\today}



\begin{abstract}
	We consider a class of stochastic evolution equations that include in particular the stochastic Camassa--Holm equation. For the initial  value problem on a torus, we first establish the local existence and uniqueness of  pathwise solutions in the  Sobolev spaces $H^s$ with $s>3/2$. Then we show that strong enough nonlinear noise can prevent blow-up almost surely.  To analyze the effects of weaker noise, we consider a  linearly multiplicative noise with 
non-autonomous pre-factor.  Then, we  formulate  precise conditions on the inital data  that lead to global existence of strong solutions or to blow-up. The blow-up
occurs as wave breaking. For blow-up with  positive probability, we derive lower bounds for these  probabilities. Finally, the blow-up rate of these solutions 
	 is precisely analyzed. 
\end{abstract}

\maketitle

%
\tableofcontents

\section{Introduction}
The Dullin--Gottwald--Holm (DGH) equation is a third-order dispersive  evolution equation given by 
\begin{eqnarray}
u_{t}-\a^2u_{xxt}+c_0 u_x+3uu_{x}+\g u_{xxx}=\a^2\left(2u_{x}u_{xx}+uu_{xxx}\right) \mbox{ in }  (0,\infty)\times \R.\label{DGH}
\end{eqnarray}
It was derived by Dullin et al.~in \cite{Dullin-Gottwald-Holm-2001-PRL}   as a model  governing planar  solutions  to 
Euler's equations in the shallow--water regime.  The unknown $u=u(t,x)$   in \eqref{DGH}  stands for the longitudinal velocity component and $\a^2,\gamma$ and $c_0$ are some physical
parameters.

The DGH equation \eqref{DGH} embeds two different integrable soliton equations. 
When  $\a=0$, \eqref{DGH} reduces to the  Korteweg--de--Vries (KdV) equation
	\begin{eqnarray}
	u_{t}+c_0 u_x+3uu_{x}+\g u_{xxx}=0, \label{KdV}
	\end{eqnarray}
while \eqref{DGH} equals to  the following Camassa--Holm (CH) equation for the choices $\g=0$ and $\a=1$,
	\begin{eqnarray}
	u_{t}-u_{xxt}+c_0 u_x+3uu_{x}=2u_{x}u_{xx}+uu_{xxx}.\label{CH}
	\end{eqnarray}
Both \eqref{KdV} and \eqref{CH} have been studied widely in the literature. 
We notice that the CH equation exhibits two interesting phenomenon, namely  (peaked) soliton interaction and wave breaking (the solution remains bounded but its slope becomes unbounded in finite time, cf. \cite{Constantin-Escher-1998-Acta}), while the KdV equation does not model breaking waves \cite{Kenig-Ponce-Vega-1993-CPAM} (when $c_0=0$, \eqref{KdV} admits a smooth soliton).  For the CH equation,   wave breaking 
   has been analyzed in \cite{Constantin-2000-JNS,Constantin-Escher-1998-CPAM,Mckean-1998-AJM} including  necessary and sufficient criterion for the occurrence of 
   breaking waves in the Cauchy problem with smooth initial data \cite{Constantin-Escher-1998-Acta,Mckean-1998-AJM}.  As pointed out in \cite{Constantin-2006-Inventiones, Constantin-Escher-2007-BAMS,Constantin-Escher-2011-Annals}, the essential feature of the CH equation
   is the occurrence of traveling waves with a peak at their crest, exactly like that the governing equations for water waves admit the so-called Stokes waves of the greatest height. 
   Bressan\&Constantin \cite{Bressan-Constantin-2007-AA,Bressan-Constantin-2007-ARMA}  developed a new approach to the analysis of the CH equation, and proved the existence of a
   global conservative and dissipative solutions. Later, Holden\&Raynaud \cite{Holden-Raynaud-2007-CPDE,Holden-Raynaud-2009-DCDS} also
   obtained  global conservative and dissipative solutions using a  Lagrangian point of view.

   Combining the linear dispersion of the KdV equation with the nonlocal dispersion of the CH equation, the DGH equation \eqref{DGH}  preserves its bi-Hamiltonian structure, is completely integrable (via the inverse scattering transform method \cite{Dullin-Gottwald-Holm-2001-PRL}) and admits  also  soliton solutions.


Here, we are interested in stochastic variants  of the  DGH equation to model   energy consuming/exchanging mechanisms in \eqref{DGH} that are driven by 
external stochastic influences. Adding multiplicative noise  has also been connected to the  prevailing hypotheses that  the onset of turbulence in fluid models involves randomness, cf. \cite{Brzezniak-Capinski-Flandoli-1991-MMMAS,Kuksin-Shirikyan-2012-book,E-2001-Chapter-CDM}. Precisely, our stochastic evolution equation writes as 
\begin{eqnarray}
u_{t}-\a^2u_{xxt}+c_0 u_x+3uu_{x}+\g u_{xxx}-\dot{W}(1-\a^2\partial^2_{xx})h(t,u)=\alpha^2\left(2u_{x}u_{xx}+uu_{xxx}\right),
\label{DGH random dissipation}
\end{eqnarray}
where $W$ is a standard 1-D Brownian motion and $h=(t,u)$  is a typically nonlinear function.
We notice that the deterministic counterpart of \eqref{DGH random dissipation} is the  weakly dissipative CH equation 
 \begin{align}
u_t-u_{xxt}+3uu_x+\lambda(1-\partial^2_{xx})h(t,u)=2u_xu_{xx}+uu_{xxx},\ \ \lambda>0.\label{dissipative CH}
\end{align}
Equation \eqref{dissipative CH} has been introduced and studied  for $h(t,u) =u$ in \cite{Lenells-Wunsch-2013-JDE,Wu-Yin-2009-JDE},
In \eqref{dissipative CH}, the operator $\lambda(1-\partial^2_{xx})$ is linear and only
models the (weak) energy dissipation. In order to model more general random energy exchange, we consider the possibly nonlinear noise term
$-\dot{W}(1-\a^2\partial^2_{xx})h(t,u)$ in \eqref{DGH random dissipation}.

 To compare our model with deterministic weakly dissipative CH type equations (see  \cite{Wu-Yin-2008-SIAM,Wu-Yin-2009-JDE,Lenells-Wunsch-2013-JDE} and the references therein), we focus our attention on the case that $\alpha\ne0$. For convenience, we assume $\a=1$ in this paper. When $\a=1$,
applying the operator $(1-\partial_{xx}^2)^{-1}$ to \eqref{DGH random dissipation} gives rise to 
the following nonlocal equation
\begin{equation}\label{SDGH-noise}
{\rm d}u+\left[\left(u-\g\right)\partial_xu+(1-\partial_{xx}^2)^{-1}\partial_x\left(u^2+\frac{1}{2}u_x^2+\left(c_0+\g\right)u\right)\right]{\rm d}t=h(t,u){\rm d}W.
\end{equation} 
In \eqref{SDGH-noise}, the operator $(1-\partial_{xx}^2)^{-1}$ in torus $\T=\R/2\pi\Z$ is understood as
\begin{align}\label{Helmboltz operator}
\left[(1-\partial_{xx}^2)^{-1}f\right](x)=[G_{\T}*f](x),\ \ G_{\T}=\frac{\cosh(x-2\pi\left[\frac{x}{2\pi}\right]-\pi)}
{2\sinh(\pi)},\ \forall\ f\in L^2(\T),
\end{align}
where $[x]$ stands for the integer part of $x$. Here we remark that for additive noise, 
\eqref{SDGH-noise} has been studied in \cite{Lv-etal-2019-JMP}. In this paper we will consider a more general context with noise   driven by a cylindrical Wiener process  $\W$, rather than a standard white noise $W$. It is assumed that $\W$ is defined on an auxiliary Hilbert space $U$ which
 is adapted to a right-continuous filtration $\{\mathcal{F}_t\}_{t\geq0}$, see Section \ref{assumptions definitions and main rsults} for more details.
 
With the above notations, the first goal of the present   paper is to  analyze   the existence and  uniqueness of pathwise solutions  and to determine 
possible blow-up criterion for  the periodic boundary value problem
	\begin{equation} \label{periodic Cauchy problem}
	\left\{\begin{aligned}
	&{\rm d}u+\left[\left(u-\g\right)\partial_xu+F(u)\right]\,{\rm d}t=h(t,u)\,{\rm d}\W,\quad x\in\T=\R/2\pi\Z, \ t>0,\\
	&u(\omega,0,x)=u_0(\omega,x),\quad x\in\T,
	\end{aligned} \right.
	\end{equation}
	where $F(u)=F_1(u)+F_2(u)+F_3(u)$ and
	\begin{align}\label{F decomposition}
\left\{\begin{aligned}
&F_1(u)=(1-\partial_{xx}^2)^{-1}\partial_x\left(u^2\right),\\
&F_2(u)=(1-\partial_{xx}^2)^{-1}\partial_x\left(\frac{1}{2}u_x^2\right),\\
&F_3(u)=(1-\partial_{xx}^2)^{-1}\partial_x\left(\left(c_0+\g\right)u\right).
	\end{aligned} \right.
	\end{align}
%
Under generic assumptions on $h(t,u)$,  we will show that \eqref{periodic Cauchy problem} has a local unique pathwise solution (see Theorem \ref{Local pathwise solution} below). Here we remark that Chen et al. in \cite{Chen-Gao-Guo-2012-JDE} have considered the stochastic CH equation with additive noise. For the linear multiplicative noise case, we refer to \cite{Tang-2018-SIMA} for the stochastic CH equation, and to \cite{Chen-Gao-2016-PA} for a stochastic modified CH equation.

 For stochastic nonlinear  evolution equations, the noise effect is a crucial question to study. Can the noise   prevent blow-up or does it even drive the formation of singularities? For example, it is known that the well-posedness of linear stochastic transport equations with noise can be established under weaker hypotheses than for its deterministic counterpart (cf.~\cite{Fedrizzi-Flandoli-2013-JFA,Flandoli-Gubinelli-Priola-2010-Invention}).  For stochastic scalar conservation laws, noise on the flux may bring some regularization effects \cite{Gess-Souganidis-2017-CPAM}, and it may also trigger the discrete entropy dissipation  in the numerical schemes for conservation laws such that the schemes enjoy some stability properties not present in the deterministic case \cite{Kroker-Rohde-2012-ANM}. 
 Moreover, we refer to   \cite{Kim-2010-JFA,GlattHoltz-Vicol-2014-AP,Rockner-Zhu-Zhu-2014-SPTA,Tang-2018-SIMA} for the dissipation of energy caused by the linear multiplicative noise. 
 
 However,  most  existing results on the regularization effects by noise for transport type equations are  for linear equations or restricted to linear growing noise. Much less is known concerning the cases of  nonlinear equations with nonlinear noises. Indeed, the interplay between regularization provided by noise and the {nonlinearities of the governing equation}  is more complicate.
For example,  singularities can be prevented in some cases (cf. \cite{Flandoli-Gubinelli-Priola-2011-SPTA}: coalescence of vortices disappears in stochastic 2D Euler equations). On the other hand, it is known  that noise does not prevent shock formation  in the Burgers equation, see \cite{Flandoli-2011-book}.  

Therefore the second goal of this work is to study the case of strong nonlinear noise  and consider its effect.    As we will see in \eqref{blow-up criterion common} below, for the solution to \eqref{periodic Cauchy problem}, its $H^s$-norm blows up if and only if its $W^{1,\infty}$-norm blows up. This suggests choosing a noise coefficient involving the ${W^{1,\infty}}$-norm of $u$. Therefore in this work we  consider the case that $h(t,u)\,{\rm d}\W
	=a\left(1+ \|u\|_{W^{1,\infty}}\right)^{\theta}u\, {\rm d}W$, where $\theta>0$, $a\in\R$ and $W$ is a standard 1-D Brownian motion. We will try to determine the range of $\theta$ and $a$ such that the solution to the following problem exists globally in time:
\begin{equation}\label{SDGH non blow up Eq}
\left\{\begin{aligned}
&{\rm d}u+\left[(u-\g)u_x+F(u)\right]{\rm d}t
= a \left(1+\|u\|_{W^{1,\infty}}\right)^{\theta}u\,{\rm d}W,\quad x\in\T, \ t>0,\\
&u(\omega,0,x)=u_{0}(\omega,x), \qquad  x\in \T.
\end{aligned} \right.
\end{equation}
As is shown in Theorem \ref{Non breaking} below,  if the noise is strong enough (either $\theta>1/2$, $a\neq0$ or $\theta=1/2$, $a^2\gg 1$), then the global existence holds true for \eqref{SDGH non blow up Eq} almost surely. This result
justifies the idea  that large nonlinear noise can actually prevent blow-up.

On the other hand, as put forward by e.g. Whitham  in \cite{Whitham-2011-Wiley},  the wave breaking phenomenon is one of the most intriguing long-standing 
problems of water wave theory. For the deterministic CH type equations, the wave breaking phenomenon has been extensively studied, see \cite{Constantin-Escher-1998-Acta,Constantin-Escher-1998-CPAM,Mckean-1998-AJM} for example. Particularly, for equations with dissipation  term $\lambda(u-u_{xx})$, we refer to \cite{Wu-Yin-2009-JDE} for the phenomenon of wave breaking. When random noise is involved, as far as we know,  we can only refer to \cite{Crisan-Holm-2018-PhyD,Tang-2020-Arxiv} for wave breaking. In \cite{Crisan-Holm-2018-PhyD} the authors proved that temporal stochasticity (in the sense of Stratonovich)
in the diffeomorphic flow map for the stochastic CH equation does not prevent the wave breaking process. In \cite{Tang-2020-Arxiv}, wave breaking in the stochastic CH equation multiplicative It\^{o} noise is considered.

Thus, the third goal of this paper is to consider noise effects  associated with the phenomenon of wave breaking. 
Due to Theorem \ref{Non breaking}, we see that if wave breaking occurs, the noise term does not grow fast. Hence we consider  $\theta=0$ in \eqref{SDGH non blow up Eq} but introduce a non-autonomous pre-factor  depending on time $t$. Precisely, we consider the DGH equation with  linear multiplicative noise
given by
	\begin{equation} \label{DGH linear noise 1}
	\left\{\begin{aligned}
	&{\rm d}u+\left[(u-\gamma)\partial_xu+F(u)\right]{\rm d}t=b(t) u \, {\rm d}W,\quad x\in\T, \ t>0,\\
	&u(\omega,0,x)=u_0(\omega,x),\quad x\in\T.
	\end{aligned} \right.
	\end{equation}
This case can be formally reformulated as the following stochastic evolution equation when $s>3$
	\begin{align}
	u_{t}-u_{xxt}+c_0 u_x+3uu_{x}+\g u_{xxx}=2u_{x}u_{xx}+uu_{xxx}+b(t)(u-u_{xx})\dot{W}.\label{DGH linear noise 2}
	\end{align}
	
When $c_0+\gamma=0$, we give two conditions on the initial data that guarantee the global existence of the solutions. Besides,  we also estimate the probability that the solution breaks and describe its breaking rate. See Theorems \ref{Decay result}-\ref{Blow-up rate}.

The precise statements of all the results above can be found in Section \ref{assumptions definitions and main rsults} jointly with the necessary assumptions on the noise coefficient.

\section{Definitions, assumptions and main results} \label{assumptions definitions and main rsults}

We begin by introducing some notations. 
$L^2(\T)$ is the usual space of square--integrable functions  on $\T$. For $s\in\R$,  $D^s=(1-\partial_{xx}^2)^{s/2}$ is defined by
$\widehat{D^sf}(k)=(1+k^2)^{s/2}\widehat{f}(k)$, where $\widehat{g}$ is the Fourier transform of $g$. The Sobolev space $H^s(\T)$ is defined as
\begin{align*}
H^s(\T)\triangleq\{f\in L^2(\T):\|f\|_{H^s(\T)}^2=\sum_{k\in{\Z}}(1+k^2)^s|\widehat{f}(k)|^2<\infty\},
\end{align*}
and the inner product $(f,g)_{H^s}$ is
$(f,g)_{H^s} :=\sum_{k\in{\Z}}(1+k^2)^s\widehat{f}(k)\cdot\overline{\widehat{g}}(k)=(D^sf,D^sg)_{L^2}.$
When the function space refers to  $\T$, we will drop $\T$ if there is no ambiguity. We will use $\lesssim $ to denote estimates that hold up to some universal \textit{deterministic} constant which may change from line to line but whose meaning is clear from the context. For linear operators $A$ and $B$, we denote  $[A,B]=AB-BA$.

We briefly recall some aspects of the theory of infinite dimensional stochastic analysis which we will use below. We refer the readers to \cite{Prato-Zabczyk-2014-Cambridge,Gawarecki-Mandrekar-2010-Springer,Kallianpur-Xiong-1995-book} for an extended treatment of this subject.

We call
$\s=(\Omega, \mathcal{F},\p,\{\mathcal{F}_t\}_{t\geq0}, \W)$ a stochastic basis, where $\{\mathcal{F}_t\}_{t\geq0}$ is a right-continuous filtration on $(\Omega, \mathcal{F})$ such that $\{\mathcal{F}_0\}$
contains all the $\p$-negligible subsets and $\W(t)=\W(\omega,t),\omega\in\Omega$ is a cylindrical  Wiener process 
adapted to $\{\mathcal{F}_t\}_{t\geq0}$. More precisely, we consider a separable Hilbert space $U$  as well as a larger Hilbert space 
$U_0$ such that the embedding $U\hookrightarrow U_0$ is Hilbert--Schmidt. 
Therefore we define
$$\W=\sum_{k=1}^\infty W_ke_k\in C([0,\infty), U_0)\ \ \p-a.s.,$$
where $\{W_k\}_{k\geq1}$ is a sequence of mutually independent one-dimensional Brownian motions and $\{e_k\}_{k\in\N}$ is a complete orthonormal basis of $U$. 

For a predictable stochastic process $G$ taking values in the space of Hilbert--Schmidt operators from $U$ to $H^s$, denoted by $L_2(U, H^s)$, 
the  It\^{o} stochastic integral 
\begin{equation*}
\int_0^\tau G{\rm d}\W=\sum_{k=1}^\infty\int_0^\tau G e_k{\rm d}W_k
\end{equation*}
is well defined (see  \cite{Prato-Zabczyk-2014-Cambridge,Prevot-Rockner-2007-book} for example).  Remember that
$$G\in L_2(U, H^s)\Longleftrightarrow\|G\|^2_{L_2(U, H^s)}=\sum_{k=1}^\infty\|G e_k\|^2_{H^s}<\infty.$$
The stochastic integral $\int_0^t G\ {\rm d}\W$ is an $H^s$-valued square--integrable martingale. In our case we have the  Burkholder-Davis-Gundy inequality 
\begin{align}\label{BDG G}
\E\left(\sup_{t\in[0,T]}\left\|\int_0^t G\ {\rm d}\W\right\|_{H^s}^p\right)
\leq C(p,s) \E\left(\int_0^T \|G\|^2_{L_2(U, H^s)}\ {\rm d}t\right)^\frac{p}{2},\ \ p\geq1.
\end{align}

\subsection{Definitions of the solutions}

We first precise the notion of pathwise solutions to \eqref{periodic Cauchy problem}.
\begin{Definition}[Pathwise solutions]\label{pathwise solution definition}
	Let $\s=(\Omega, \mathcal{F},\p,\{\mathcal{F}_t\}_{t\geq0}, \W)$ be a fixed stochastic basis. Let $s>3/2$ and $u_0$ be an $H^s$-valued $\mathcal{F}_0$ measurable random variable.
	\begin{enumerate}
		\item A local pathwise solution to \eqref{periodic Cauchy problem} is a pair $(u,\tau)$, where $\tau$ is a stopping time satisfying $\p\{\tau>0\}=1$ and
		$u:\Omega\times[0,\infty)\rightarrow H^s$  is an $\mathcal{F}_t$ predictable $H^s$-valued process satisfying
		\begin{equation*}
		u(\cdot\wedge \tau)\in C([0,\infty);H^s)\ \ \p-a.s.,
		\end{equation*}
		and for all $t>0$,
		\begin{equation*} 
		u(t\wedge \tau)-u(0)+\int_0^{t\wedge \tau}
		\left[(u-\gamma)\partial_xu+F(u)\right]{\rm d}t'
		=\int_0^{t\wedge \tau}h(t',u){\rm d}\W \ \ \p-a.s.
		\end{equation*}
		\item Local pathwise solutions are said to be pathwise unique, if given any two pairs of local pathwise solutions $(u_1,\tau_1)$ and $(u_2,\tau_2)$ with $\p\left\{u_1(0)=u_2(0)\right\}=1,$ we have
		\begin{equation*}
		\p\left\{u_1(t,x)=u_2(t,x),\ \forall\ (t,x)\in[0,\tau_1\wedge\tau_2]\times \T\right\}=1.
		\end{equation*}
		\item Additionally, $(u,\tau^*)$ is called a maximal solution to \eqref{periodic Cauchy problem} if $\tau^*>0$ almost surely and if there is an increasing sequence $\tau_n\rightarrow\tau^*$ such that for any $n\in\N$, $(u,\tau_n)$ is a pathwise solution to \eqref{periodic Cauchy problem} and on the set $\{\tau^*<\infty\}$,
		\begin{equation*} 
		\sup_{t\in[0,\tau_n]}\|u\|_{H^s}\geq n.
		\end{equation*}
		\item If $\tau^*=\infty$ almost surely, then we say that the pathwise solution exists globally.
	\end{enumerate}
\end{Definition}

\subsection{Assumptions} 
Next, we prescribe some conditions on the noise coefficient 
$h$ in \eqref{periodic Cauchy problem} and on  $b$ in 
\eqref{DGH linear noise 2}.

\begin{Assumption}\label{Assumption-1} Let $s>1/2$.
	We assume that $h:[0,\infty)\times H^s\ni (t,u)\mapsto h(t,u)\in L_2(U, H^s)$ is continuous in $(t,u)$. Moreover, we assume the following:
	\begin{itemize}

		\item There is a non-decreasing locally bounded function $f(\cdot):[0,\infty)\rightarrow[0,\infty)$ such that for any $t>0$,
		\begin{align}\label{assumption 1 for h}
		\|h(t,u)\|_{L_2(U, H^s)}\leq f(\|u\|_{W^{1,\infty}})\|u\|_{H^s}.
		\end{align}
		Particularly, in the additive noise case, we assume  $h:[0,\infty)\times \T\ni (t,x)\mapsto h(t,x)\in L_2(U, H^s)$ is continuous meaning that \eqref{assumption 1 for h} reduces to $\|h(t,x)\|_{L_2(U, H^s)}\leq C$ for some $C>0$.
		\item
		There is a non-decreasing  locally bounded function $q(\cdot):[0,\infty)\rightarrow[0,\infty)$, such that for any $t>0$,
		\begin{equation}\label{assumption 2 for h}
		\sup_{\|u\|_{H^s},\|v\|_{H^s}\le N}\left\{{\bf 1}_{\{u\ne v\}}  \frac{\|h(t,u)-h(t,v)\|_{L_2(U, H^s)}}{\|u-v\|_{H^s}}\right\} \le q(N),\ \ N\ge 1.
		\end{equation}
	\end{itemize}
\end{Assumption}

 After the regularization effect of strong noise is established in Theorem \ref{Non breaking}, to  analyze the effect of 
  noise on the regularity of pathwise solutions, we restrict ourselves to the linear-noise case   \eqref{DGH linear noise 2} imposing the following bounds on the coefficient $b$.

\begin{Assumption}\label{Assumption-3}
	When considering \eqref{DGH linear noise 2} with non-autonomous linear noise $b(t) u\, {\rm d}  W$, we assume that there are constants $b_*,b^*>0$ such that $\displaystyle0<b_*\leq b^2(t)\leq b^*$ for all $t$.

\end{Assumption}

\begin{Remark}\label{Remark assumption explanation}
Let us give some 
	  brief explanations for  the assumptions. 
	  \begin{itemize}
	  	\item  The function  $h:[0,\infty)\times H^s\ni (t,u)\mapsto h(t,u)\in L_2(U, H^s)$ is required to be continuous in $(t,u)$. This will be essential to pass to the  limit when establishing the existence of  a martingale solution {as an intermediate step}, cf. \cite{Breit-Feireisl-Hofmanova-2018-Book,Tang-2018-SIMA,Tang-2020-Arxiv}.
	  	
	  	\item  The uniform-in-time assumption  \eqref{assumption 1 for h}  bounds the growth of the   $L_2(U, H^s)$-norm of the  noise coefficient in terms of  a product  of a nonlinear  function of the  $W^{1,\infty}$-norm and the $H^s$-norm. This allows us  to  control the  $W^{1,\infty}$-norm by some cut-off later. 
	  	
	  	\item  Formula  \eqref{assumption 2 for h} ensures local Lipschitz continuity in $H^s$, which will be used to obtain pathwise uniqueness.
	  	
	  	\item   Let us outline that we will use a Girsanov-type transformation to study \eqref{DGH linear noise 1} (see Remark \ref{girsanov} and Section \ref{global existence and wave breaking}). The assumption $b^2(t)\leq b^*$ is used to guarantee that such transformation is well-defined and the condition $b(t)\neq0, t\ge0$ is needed to establish certain estimates for the Girsanov-type process ${\rm e}^{\int_0^tb(t') {\rm d} W_{t'}-\int_0^t\frac{b^2(t')}{2} {\rm d}t'}$ (see Lemma \ref{eta Lemma}).  In Theorem \ref{Decay result}, the  condition $0<b_*\leq b^2(0)$ is used to bound the initial data, cf. \eqref{bound Hs global}.
	  \end{itemize}

\end{Remark}

\subsection{Main results and remarks}
Now we present our results. For the general case \eqref{periodic Cauchy problem}, we have the following local existence result which moreover relates the possible blow-up in the $H^s$-norm to simultaneous blow up in the $W^{1,\infty}$-norm.

\begin{Theorem}[Maximal solutions]\label{Local pathwise solution}
	Let $s>3/2$, $c_0,\g\in\R$ and let $h(t,u)$ satisfy Assumption \ref{Assumption-1}. For a given stochastic basis $\s=(\Omega, \mathcal{F},\p,\{\mathcal{F}_t\}_{t\geq0}, \W)$ and  an $H^s$-valued $\mathcal{F}_0$ measurable random variable $u_0$ satisfying $\E\|u_0\|^2_{H^s}<\infty$, the initial value problem \eqref{periodic Cauchy problem} admits a local unique pathwise solution $(u,\tau)$  in the sense of Definition \ref{pathwise solution definition} with
	\begin{equation}\label{L2 moment bound}
	u(\cdot\wedge \tau)\in L^2\left(\Omega; C\left([0,\infty);H^s\right)\right).
	\end{equation}
	Besides, $(u,\tau)$ can be extended to a unique maximal solution $(u,\tau^*)$ in the sense of Definition \ref{pathwise solution definition} and the following blow-up criterion holds true:
	\begin{equation}\label{blow-up criterion common}
	\textbf{1}_{\left\{\limsup_{t\rightarrow \tau^*}\|u(t)\|_{H^{s}}=\infty\right\}}=\textbf{1}_{\left\{\limsup_{t\rightarrow \tau^*}\|u(t)\|_{W^{1,\infty}}=\infty\right\}}\ \ \p-a.s.
	\end{equation}
\end{Theorem}

\begin{Remark}
 For the  proof of Theorem \ref{Local pathwise solution} one can follow the ideas in  e.g.~\cite{GlattHoltz-Vicol-2014-AP,Debussche-Glatt-Temam-2011-PhyD,Breit-Feireisl-Hofmanova-2018-CPDE,Breit-Hofmanova-2016-IUMJ,Breit-Feireisl-Hofmanova-2018-Book,Tang-2018-SIMA,Crisan-Flandoli-Holm-2018-JNS}
	by constructing a sequence of approximations for a problem with cut-off
	for the  $W^{1,\infty}$-norm.
	Such a cut-off  implies at-most linear  growth of  $u$ and 
	  guarantees  the global existence of an approximate solution. 
	  Otherwise we have  to find a  positive lower bound for the lifespan of the approximate solutions, which is {\it a priori} not clear. Besides, with the cut-off, one can close   the {\it a priori} $L^2(\Omega;H^{s})$ estimate by splitting $\E(\|u\|_{H^s}^2\|u\|_{W^{1,\infty}})$.
	  
\end{Remark}

Turning to noise-driven  regularization effects,  the blow-up  criterion \eqref{blow-up criterion common} suggests relating the noise coefficient  to  the  ${W^{1,\infty}}$-norm of $u$.  
Therefore we  consider \eqref{SDGH non blow up Eq} with {scalable} noise impact, i.e.  we  
 assume $h(t,u)=a (1+\|u\|)^{\theta}_{W^{1,\infty}}u$ for some $\theta>0$ and $a\in\R$. 
 When $a$ and $\theta$ satisfy certain strength-conditions, the noise term  counteracts the 
formation of singularities and we have

\begin{Theorem}[Global existence for strong  nonlinear noise]\label{Non breaking} Let $\s=(\Omega, \mathcal{F},\p,\{\mathcal{F}_t\}_{t\geq0}, W)$ be a fixed stochastic basis. Let $s>\frac{5}{2}$, $c_0,\g\in\R$ and $u_0\in H^s$ be an $H^s$-valued $\mathcal{F}_0$-measurable random variable with $\E\|u_0\|^2_{H^s}<\infty$.  Assume that $\theta$ and $a$ satisfy 
	\begin{equation}\label{a theta condition}
 \text{either} \ a\in\R\setminus\{0\},\ \theta>\frac{1}{2}
\   \text{or}  \  
a^2> 2Q\ ,\theta=\frac{1}{2},
	\end{equation}
	where $Q=Q(s,c_0,\g)$ is a constant that will be specified  in Lemma \ref{uux+F Hs inner product}. Then 
	the corresponding maximal solution
	$(u,\tau^*)$  to \eqref{SDGH non blow up Eq} satisfies
	$$\p\left\{	\tau^*=\infty\right\}
	=1.$$
\end{Theorem}

 Theorem \ref{Non breaking} implies that  blow-up of pathwise solutions might only be observed if the noise is weak.  To detect such noise, we analyze  the simpler 
ansatz $h(t,u)= b(t)u$ as in 
 \eqref{DGH linear noise 1}.  Even in this linear noise case the situation is quite subtle allowing for  global existence as well as  blow-up of solutions. For global existence, we can identify two cases.

\begin{Theorem}[Global existence for weak noise I]\label{Decay result} Let $s>3/2$. Assume  $c_0+\g=0$. Let $b(t)$ satisfy Assumption \ref{Assumption-3}  and $\s=(\Omega, \mathcal{F},\p,\{\mathcal{F}_t\}_{t\geq0}, W)$ be a fixed stochastic basis. Assume $u_0$ is an $H^s$-valued $\mathcal{F}_0$ measurable random variable satisfying $\E\|u_0\|^2_{H^s}<\infty$. Let $K=K(s)>0$ be a constant such that the embedding  $\|\cdot\|_{W^{1,\infty}}<K\|\cdot\|_{H^s}$ holds. Then there is a $C=C(s)>1$ such that for any $R>1$ and $\lambda_1>2$, if 
	\begin{equation}\label{bound Hs global}
	\|u_0\|_{H^s}<\frac{b_*}{CK\lambda_1 R}\ \ \p-a.s.,
	\end{equation}
then  \eqref{DGH linear noise 1} has a maximal solution $(u,\tau^*)$ satisfying  for any $\lambda_2>\frac{2\lambda_1}{\lambda_1-2}$ the estimate
	\begin{equation}\label{Decay statement}
	\p \left\{
	\|u(t)\|_{H^s}<\frac{b_*}{CK\lambda_1}
	{\rm e}^{-\frac{\left((\lambda_1-2)\lambda_2-2\lambda_1\right)}{2\lambda_1\lambda2}\int_0^tb^2(t') {\rm d}t'}
	\ {\rm\ for\ all}\ t>0
	\right\}
	\geq 1-\left(\frac{1}{R}\right)^{2/\lambda_2}.
	\end{equation}
	
\end{Theorem}

\begin{Theorem}[Global existence for weak noise II]\label{Global existence result} Let $c_0+\g=0$ and $s>3$. Let $b(t)$ satisfy Assumption \ref{Assumption-3}  and $\s=(\Omega, \mathcal{F},\p,\{\mathcal{F}_t\}_{t\geq0}, W)$ be a fixed stochastic basis. Assume $u_0$ is an $H^s$-valued $\mathcal{F}_0$ measurable random variable satisfying $\E\|u_0\|^2_{H^s}<\infty$.	If  $u_0$ satisfies
	\begin{equation*} 
	\p\left\{(1-\partial_{xx}^2)u_0(x)>0,\ \forall x\in\T\right\}=p,\ \
	\p\left\{(1-\partial_{xx}^2)u_0(x)<0,\ \forall x\in\T\right\}=q,
	\end{equation*}
	for some $p,q\in[0,1]$, then the corresponding maximal solution $(u,\tau^*)$ to \eqref{DGH linear noise 1} satisfies
	\begin{equation*}
	\p\{\tau^*=\infty\}\geq p+q.
	\end{equation*}
\end{Theorem}

\begin{Remark}

	Theorem \ref{Decay result} provides a global existence result for  initial data with bounded $H^s$-norm  depending on the strength of the noise. This result can \textit{not} be observed in the  deterministic case because $ 0<b_*\leq b^2(t) $ is required  (see Assumption \ref{Assumption-3}). On the other hand, since the proof of Theorem \ref{Global existence result} relies on the analysis of a PDE with random coefficient (see \eqref{periodic Cauchy problem transform} below),
	the deterministic case can be included by formally setting the random coefficient equal  to $1$. Therefore, in this sense, Theorem \ref{Global existence result} covers the corresponding deterministic result, cf. \cite{Liu-2006-MA,Constantin-Escher-1998-CPAM}. Indeed, by letting $\beta\equiv1$ in \eqref{periodic Cauchy problem transform} and taking $(p,q)=(1,0)$ or $(p,q)=(0,1)$ in Theorem \ref{Global existence result}, we
	obtain the global existence for the deterministic DGH equation.

\end{Remark}

According to \eqref{blow-up criterion common} in Theorem \ref{Local pathwise solution}, a blow-up comes along with an explosion of the $W^{1,\infty}$-norm. For the 
special noise in \eqref{DGH linear noise 1} we can improve the result by showing that a blow-up  is related to the first spatial derivative only  and 
corresponds to the wave-breaking phenomenon with exploding negative slope.

\begin{Theorem}[Blow-up scenario]\label{blow-up criterion weak noise}
	Let $c_0+\g=0$, $s>3$ and Assumption \ref{Assumption-3} be satisfied.  Let $\s=\left(\Omega, \mathcal{F},\p,\{\mathcal{F}_t\}_{t\geq0}, W\right)$ be fixed in advance. Let $(u,\tau^*)$ be  the  unique maximal solution to \eqref{DGH linear noise 1} starting from  an $\mathcal{F}_0$ measurable random variable $u_0\in L^2(\Omega;H^s)$.
	 Then the singularities   can arise only in the form of wave breaking, i.e.
	\begin{equation}\label{blow-up by WB 1}
	\p\left\{
	\|u(t)\|_{L^\infty}\lesssim A\|u_0\|_{H^{1}}<\infty,\ \forall \, t>0
	\right\}=1,
	\end{equation}
	where $A=A(\omega)=\sup_{t>0}{\rm e}^{\int_0^tb(t') \ {\rm d} W_{t'}-\int_0^t\frac{b^2(t')}{2} \ {\rm d}t'}<\infty$ $\p-a.s.$, and 
	\begin{equation}\label{blow-up by WB 2}
\textbf{1}_{\left\{\limsup_{t\rightarrow \tau^*}\|u(t)\|_{H^{s}}=\infty\right\}}=\textbf{1}_{\left\{\liminf_{t\rightarrow \tau^*}\left[\min_{x\in\T}u_x(t,x)\right]=-\infty\right\}}\ \ \p-a.s.
	\end{equation}
\end{Theorem}


 Still we have not identified initial data  for  \eqref{DGH linear noise 1} that lead to a blow-up. A precise condition in terms of probability is given in the 
next theorem. To formulate it, we introduce the number $\lambda >0$
  such that for any $f\in H^3$,  the estimate 
	\begin{equation}	\label{f2<f H1} 
	\max_{x\in\T}f^2(x)\leq  \lambda\|f\|^2_{H^1}
	\end{equation}
holds. 

\begin{Theorem}[Wave breaking and its probability]\label{Wave breaking} Let $\s=(\Omega, \mathcal{F},\p,\{\mathcal{F}_t\}_{t\geq0}, W)$ be a fixed stochastic basis, $c_0+\g=0$ and $s>3$. Let Assumption \ref{Assumption-3} be verified and let $u_0\in L^2(\Omega;H^s)$ be  $\mathcal{F}_0$ measurable. If for some $c\in (0,1)$,
	$$\min_{x\in\T}\partial_xu_0(x)<-\frac{1}{2}\sqrt{\frac{(b^*)^2}{c^2}+4\lambda\|u_0\|_{H^1}^2}-\frac{b^*}{2c}\ \ \p-a.s.,$$
	where $b^*$ is given in Assumption \ref{Assumption-3} and $\lambda$ is given in \eqref{f2<f H1},
	then the maximal solution $(u,\tau^*)$ to \eqref{DGH linear noise 1} $($or \eqref{DGH linear noise 2}, equivalently$)$ satisfies
	$$\p\left\{
	\tau^*<\infty
	\right\}\geq\p\left\{{\rm e}^{\int_0^t b(t') \, {\rm d}W_{t'}}>c \, \,  \forall\,  t>0\right\}>0.$$
	By Theorem \ref{blow-up criterion weak noise}, we have $\displaystyle\p\left\{u\ {\rm  breaks\ in\ finite\ time}\right\}\geq\p\left\{{\rm e}^{\int_0^t b(t')\, {\rm d}W_{t'}}>c\,\, \forall\,t>0\right\}>0.$
\end{Theorem}

\begin{Remark}
	Whereas Theorem \ref{Decay result} provides a global existence result, Theorem \ref{Wave breaking} detects the formation of singularities in finite time under certain conditions on the initial data.  We stress that these two conditions are mutually exclusive.
	  In Theorem \ref{Decay result} we suppose  $\|u_0\|_{H^s}\leq\frac{b_*}{CK\lambda_1 R}$ with $C>1,\lambda_1>2$ and $R>1$ almost surely.  Then $u$ satisfies \eqref{Decay statement}. In Theorem \ref{Wave breaking} we suppose for some $c\in (0,1)$ that  $\min_{x\in\T}\partial_xu_0(x)<-\frac{1}{2}\sqrt{\frac{(b^*)^2}{c^2}+4\lambda\|u_0\|_{H^1}^2}-\frac{b^*}{2c}<-\frac{b^*}{2c}$ almost surely holds. But this means   $\|u_0\|_{H^s}>\frac{1}{K}\|u_0\|_{W^{1,\infty}}\geq\frac{1}{K}|\min_{x\in\T}\partial_xu_0(x)|> \frac{b^*}{2cK}>\frac{b_*}{CK\lambda_1 R}$.
\end{Remark}

We conclude this section with a result refining  Theorem \ref{blow-up criterion weak noise}. It is possible to quantify 
the blow-up rate.
\begin{Theorem}[Wave breaking rate]\label{Blow-up rate}
	Let the conditions in Theorem \ref{blow-up criterion weak noise} hold true. Then
	\begin{equation}\label{Blow-up rate equation}
	\lim_{t\rightarrow \tau^*} \left( \min_{x\in\T}[u_x(t,x)]\int_{t}^{\tau^*}\beta(t')\ {\rm d}t'\right) =-2\beta(\tau^*) 
	\ \ \text{a.e.\ on}\ \ \{\tau^*<\infty\},
	\end{equation}
	where 
	\begin{equation*}\label{beta}\beta(\omega,t)={\rm e}^{\int_0^tb(t') {\rm d} W_{t'}-\int_0^t\frac{b^2(t')}{2} {\rm d}t'}.\end{equation*}
\end{Theorem}

\begin{Remark}\label{rem_deterministic}
	As a corollary of  Theorems \ref{blow-up criterion weak noise}  and  \ref{Blow-up rate}, we have  that as long as singularities occurs, they can arise only in the form of wave breaking and the breaking rate is given  by \eqref{Blow-up rate equation}. This result is optimal in the sense that it is consistent with the result for the corresponding 
	deterministic case. Indeed, for the deterministic DGH equation, the blow-up rate is, \cite[Theorem 4.2]{Liu-2006-MA},
	\begin{equation*}
	\lim_{t\rightarrow \tau^*} \left( \min_{x\in\T}[u_x(t,x)](\tau^*-t)\right)
	=-2.
	\end{equation*}
	Formally, since the deterministic DGH equation can be viewed as \eqref{periodic Cauchy problem transform} with $\beta\equiv1$, we see that the blow-up estimate \eqref{Blow-up rate equation} coincides with the above deterministic result when $\beta\equiv1$. 
\end{Remark}

\begin{Remark}\label{girsanov}
 Let us make a comment on the idea for the subsequent analysis of  
	 \eqref{DGH linear noise 1} which is 
	motivated by \cite{GlattHoltz-Vicol-2014-AP,Rockner-Zhu-Zhu-2014-SPTA,Tang-2018-SIMA}. 
	 By introducing the Girsanov-type transformation
	\begin{align*}
	v=\frac{1}{\beta(\omega,t)} u,\ \
	\beta(\omega,t)={\rm e}^{\int_0^tb(t') {\rm d} W_{t'}-\int_0^t\frac{b^2(t')}{2} {\rm d}t'},
	\end{align*}
	we obtain an equation for $v$ (see Section \ref{global existence and wave breaking} for the detailed calculation), namely
	$$
	v_t+\beta vv_x-\g v_x+\beta (1-\partial_{xx}^2)^{-1}\partial_x\left(v^2+\frac{1}{2} v_x^2\right) +(c_0+\g)(1-\partial_{xx}^2)^{-1}\partial_xv=0.$$
Although  the above equation for $v$ does not depend on a stochastic integral on $v$ itself, to extend the deterministic results to the stochastic setting, we need to overcome a few technical difficulties since the system is not only random but also non-autonomous (see e.g., \eqref{global time tau}, \eqref{tau 2 Girsanov} and \eqref{tau 1 > tau 2}). With the help of certain estimates and asymptotic limits of  Girsanov-type processes (see Lemma \ref{eta Lemma}), we are able to apply the energy estimate pathwisely (for a.e. $\omega\in\Omega$) to study the global existence and possible blow-up of  solutions.
\end{Remark}

We outline the rest of the paper. In the next section, we briefly recall some relevant preliminaries. In Section \ref{Local well-posedness}, we prove Theorem \ref{Local pathwise solution}.  For the large noise case, we prove Theorem \ref{Non breaking} in Section \ref{Noise VS blow-up}. For the non-autonomous linear multiplicative noise case, we consider the global existence, decay, wave breaking and the blow-up rate of the pathwise solutions and prove Theorems \ref{Decay result}, \ref{Global existence result}, \ref{Wave breaking} and \ref{Blow-up rate} in Section \ref{global existence and wave breaking}.

\section{Preliminary Results}\label{Preliminary Results}
We summarize  some auxiliary results, which will be used to prove our main results from Section \ref{assumptions definitions and main rsults}. 
Define the regularizing operator $T_\e$ on $\T$ as
\begin{equation}\label{Define Te}
T_\e f(x):=(1-\e^2 \Delta)^{-1}f(x)= \sum_{k\in\Z^d} \left(1+\e^2 |k|^2\right)^{-1}  \widehat{f}(k)\, {\rm e}^{{\rm i}xk},\ \ \e\in(0,1).
\end{equation}
Since $T_\e$ can be characterized by its Fourier multipliers, it is  easy to  see  we have 
\begin{align}
[D^s,T_{\varepsilon}]=0,\label{mollifier property 3}
\end{align}
\begin{align}
(T_{\varepsilon}f, g)_{L^2}&=(f, T_{\varepsilon}g)_{L^2},\label{mollifier property 4}
\end{align}
\begin{align}
\|T_{\varepsilon}u\|_{H^s}&\leq \|u\|_{H^s}.\label{mollifier property 5}
\end{align}
Furthermore, we have  
\begin{Lemma}[\cite{Tang-2020-Arxiv}]\label{Te commutator} 
	Let $f,g:\T\rightarrow\R$ such that $g\in W^{1,\infty}$ and $f\in L^2$. Then for some $C>0$,
	\begin{align*}
	\|[T_{\varepsilon}, (g\cdot \nabla)]f\|_{L^2}
	\leq C\|g\|_{W^{1,\infty}}\|f\|_{L^2}.
	\end{align*}
\end{Lemma}
The following estimates are classical for Sobolev spaces.
\begin{Lemma}[\cite{Kato-Ponce-1988-CPAM}]\label{Kato-Ponce commutator estimate} Let $s>1$.
	There is a $C_s>0$ such that for all $f\in H^s\cap W^{1,\infty},\ g\in H^{s-1}\cap L^{\infty}$ we have
	$$
	\|\left[D^s,f\right]g\|_{L^2}\leq C_s \big(\|D^sf\|_{L^2}\|g\|_{L^{\infty}}+\|\partial_xf\|_{L^{\infty}}\|D^{s-1}g\|_{L^2}\big).
	$$
\end{Lemma}
\begin{Lemma}[\cite{Kato-Ponce-1988-CPAM}]
	\label{Moser estimate}
	Let $s>0$, then there is a $C_s>0$ such that we have for all $f,g \in H^s\cap L^{\infty}$ the estimate 
	$$\|fg\|_{H^s}\leq C_s \big(\|f\|_{H^s}\|g\|_{L^{\infty}}+\|f\|_{L^{\infty}}\|g\|_{H^s} \big).$$
\end{Lemma}

Specifically, for our problem \eqref{periodic Cauchy problem}, we have introduced the
  nonlocal term $F(\cdot)$ in \eqref{F decomposition}. Using the Moser estimate from Lemma \ref{Moser estimate}, we can obtain the next statement on 
  $F(\cdot)$ (see \cite{Tang-Zhao-Liu-2014-AA}).
\begin{Lemma}\label{F(u) lemma}
	For  $F(\cdot)$ defined in \eqref{F decomposition} and for any $v_1,v_2\in H^s$ with $s>3/2$, we have
	\begin{align*}
	\|F(v)\|_{H^s}&
	\lesssim \left(\|v\|_{L^\infty}+\|\partial_xv\|_{L^\infty}+(c_0+\g)\right)\|v\|_{H^s},
	\ \ s>3/2,\\
	\|F(v_1)-F(v_2)\|_{H^s}&
	\lesssim\left(\|v_1\|_{H^s}+\|v_2\|_{H^s}+(c_0+\g)\right)\|v_1-v_2\|_{H^s},
	\ \ s>3/2,\\
	\|F(v_1)-F(v_2)\|_{H^s}&
	\lesssim\left(\|v_1\|_{H^{s+1}}+\|v_2\|_{H^{s+1}}+(c_0+\g)\right)\|v_1-v_2\|_{H^s},
	\ \ 3/2>s>1/2.
	\end{align*}
\end{Lemma}

The following estimate will be used in the proof of the blow-up criterion \eqref{blow-up criterion common} and of Theorem \ref{Non breaking}.
\begin{Lemma}\label{uux+F Hs inner product}
Let $s>3/2$, $c_0,\g\in\R$. Let $F(\cdot)$ and $T_\e$ be given in \eqref{F decomposition} and \eqref{Define Te}, respectively. 
There is a constant $Q=Q(s,c_0,\g)>0$ such that for all $\e>0$,
\begin{equation*}
\left|\left(T_\e \left[(u-\g)u_x\right], T_\e u\right)_{H^s}\right|
+\left|\left(T_\e F(u), T_\e u\right)_{H^s}\right|\leq Q\left(1+\|u\|_{W^{1,\infty}}\right)\|u\|^2_{H^s}.
\end{equation*}
\end{Lemma}
\begin{proof}
We first notice that
\begin{align*}
\left(T_\e \left[(u-\g)u_x\right], T_\e u\right)_{H^s}=\int_{\T}D^s T_\e \left[(u-\g)u_x\right] \cdot D^s T_\e u\ {\rm d}x
=\int_{\T}D^s T_\e \left[uu_x\right] \cdot D^s T_\e u\ {\rm d}x.
\end{align*}
Due to  \eqref{mollifier property 3} and \eqref{mollifier property 4}, we commute the operator  to derive
\begin{align*}
&\left(D^sT_\e 
\left[uu_x\right],D^sT_\e u\right)_{L^2}\notag\\
=&\left(\left[D^s,u\right]u_x,D^sT^2_\e u\right)_{L^2}+
\left([T_\e,u]D^su_x, D^sT_\e u\right)_{L^2}
+\left(uD^sT_\e u_x, D^sT_\e u\right)_{L^2}.
\end{align*}
Then it follows from Lemmas   \ref{Te commutator} and \ref{Kato-Ponce commutator estimate}, integration by parts, \eqref{mollifier property 5} and $H^s\hookrightarrow W^{1,\infty}$ that
\begin{equation*}
\left|\left(T_\e \left[(u-\g)u_x\right], T_\e u\right)_{H^s}\right|
\lesssim
\|u\|_{W^{1,\infty}}\|u\|^2_{H^s}.
\end{equation*}
Using Lemma \ref{F(u) lemma} and \eqref{mollifier property 5} directly, we have 
\begin{equation*}
\left|\left(T_\e F(u), T_\e u\right)_{H^s}\right|
\lesssim
\left(\|u\|_{W^{1,\infty}}+(c_0+\g)\right)\|u\|^2_{H^s}.
\end{equation*}
Combining the above two inequalities gives rise to the desired estimate of the lemma.
\end{proof}

The following lemmata have been established for the real-line case in \cite{Constantin-Escher-1998-Acta} and \cite{Constantin-2000-JNS}, respectively. They hold likewise for   $x\in\T$, using the periodicity on $\T$.
\begin{Lemma}\label{Constantin-Escher}
	Let $T >0$ and $u\in C^1([0,T); H^2(\T))$. Then given any $t\in[0,T)$, there is at least one point $z(t)$ with
	\begin{equation*}
	M(t):=\min_{x\in\T}[u_x(t,x)]=u_x(t,z(t)).
	\end{equation*}
	Moreover, the function $M=M(t)$ is almost everywhere differentiable on $(0,T)$ with
	\begin{equation*}
	\frac{{\rm d}}{{\rm d}t}M(t)=u_{tx}(t,z(t))\ \ {\rm a.e.\ on}\ (0,T).
	\end{equation*}
\end{Lemma}

We conclude this preparatory section with some results from \cite{Rohde-Tang-2020-JDDE}, which are needed to establish the 
theorems on global existence.

\begin{Lemma}\label{eta Lemma}
	Let Assumption \ref{Assumption-3} hold true and assume that $a(t)\in C([0,\infty))$ is a bounded function. For 
	\begin{equation*}
	X={\rm e}^{\int_0^tb(t') \, {\rm d} W_{t'}+\int_0^ta(t')-\frac{b^2(t')}{2} {\rm d}t'}
	\end{equation*}
	 the following properties hold true.
	
	\begin{enumerate}[label={ {\rm (\roman*)}}]
		
		\item\label{eta->0}	Let $\phi(t):=\int_0^tb^2(t') \,{\rm d}t'$  with inverse $\phi^{-1}(t)$.
		If 	\begin{equation*}
		\limsup_{t \to \infty} \frac{1}{\sqrt{2 t \log \log t}} \biggl(\int_0^{\phi^{-1}(t)} a(t') \, {\rm d}t' 
		- \frac{t}{2}\biggr) < -1 ,
		\end{equation*}
		then
		\begin{equation*}
		\lim_{t \to \infty}X(t)=0 \ \ \p-a.s.
		\end{equation*}
		If
		\begin{equation*}
		\liminf_{t \to \infty} \frac{1}{\sqrt{2 t \log \log t}} \biggl(\int_0^{\phi^{-1}(t)} a(t') \, {\rm d}t' - \frac{t}{2}\biggr)>1 ,
		\end{equation*}
		then
		\begin{equation*}
		\lim_{t \to \infty}X(t)=\infty \ \ \p-a.s.
		\end{equation*}

		\item\label{exit time eta} Let  $a(t)=\lambda b^2(t)$ with $\lambda<\frac12$ and $\tau_{R}=\inf \{t \geq 0 : X(t)>R\}$ with $R>1$, then
		\begin{equation*}
		\mathbb{P}\left(\tau_{R}=\infty\right) \geq 1-\left(\frac{1}{R}\right)^{1-2\lambda}.
		\end{equation*}
	\end{enumerate}
\end{Lemma}

\section{Proof of Theorem \ref{Local pathwise solution}}
\label{Local well-posedness}
We consider the initial value problem \eqref{periodic Cauchy problem}.
The proof of existence and uniqueness of pathwise solutions can be carried out by standard procedures as used in many works, see  \cite{GlattHoltz-Vicol-2014-AP,GlattHoltz-Ziane-2009-ADE,Tang-2018-SIMA,Tang-2020-Arxiv,Rohde-Tang-2020-JDDE,Breit-Feireisl-Hofmanova-2018-Book,Breit-Feireisl-Hofmanova-2018-CPDE} for more details. Therefore we only give a sketch.

\begin{enumerate}
	\item Firstly, one constructs a suitable approximation scheme  using a cut-off function to control the  $W^{1,\infty}$-norm 
	(arising from $(u-\gamma)u_x$, \eqref{assumption 1 for h}  and Lemma \ref{F(u) lemma}). With such cut-off, both the drift and diffusion coefficients in the problem become locally Lipschitz continuous and  grow linearly in $u$ (cf. \cite{Tang-2018-SIMA,Tang-2020-Arxiv,Rohde-Tang-2020-JDDE}). Thus the approximation solutions exist globally.  Besides,  such cut-off enables us to close   the {\it a priori} $L^2(\Omega;H^{s})$ estimate by splitting $\E(\|u\|_{H^s}^2\|u\|_{W^{1,\infty}})$. Therefore by using Lemma \ref{Kato-Ponce commutator estimate} and \eqref{assumption 1 for h}, uniform estimates for the approximation solutions can be established. We refer the readers to  \cite{Tang-2018-SIMA,Tang-2020-Arxiv} for some closely related  models;
	
	\item  Secondly, by the uniform estimates, one obtains the tightness of the distributions of the approximation solution in $\Pm{}\left(C([0,T];H^{s-1})\right)$, where $\Pm{}\left(C([0,T];H^{s-1})\right)$ is the collection of Borel probability measures on $C([0,T];H^{s-1})$. We refer to \cite{Tang-2018-SIMA,Tang-2020-Arxiv,Ren-Tang-Wang-2020-Arxiv} for example.
	Applying the probabilistic compactness arguments, i.e., the Prokhorov theorem and the Skorokhod theorem, and using some technical convergence results as in \cite{Bensoussan-1995-AAM,Debussche-Glatt-Temam-2011-PhyD,Breit-Feireisl-Hofmanova-2018-Book,Breit-Feireisl-Hofmanova-2018-CPDE}, one verifies the existence of a martingale solution in $H^s$ with $s>3$. In this step $s>3$ is an intermediate requirement  because the convergence is in $H^{s-1}$ and we need to control the $W^{1,\infty}$-norm by the embedding $H^{s-1}\hookrightarrow W^{1,\infty}$;

	\item Thirdly, by Lemma \ref{F(u) lemma} and \eqref{assumption 2 for h}, one can show that pathwise uniqueness holds. Then 
	the Gy\"ongy--Krylov characterization of the convergence in probability (see  \cite{Gyongy-Krylov-1996-PTRF}) can be applied to show the existence and uniqueness of a pathwise solution in $H^s$ with $s>3$, cf. \cite{GlattHoltz-Vicol-2014-AP,Breit-Feireisl-Hofmanova-2018-Book,Tang-2020-Arxiv};
	
	\item Finally,  mollifying initial data, analyzing the convergence and employing the  argument as in \cite{GlattHoltz-Ziane-2009-ADE,GlattHoltz-Vicol-2014-AP,Tang-2018-SIMA,Tang-2020-Arxiv} lead to a local pathwise solution $(u,\tau)$ to \eqref{SDGH-noise} with 
	  $u(\cdot\wedge\tau)\in L^2\left(\Omega; C\left([0,\infty);H^s\right)\right)$  for $u_0\in L^{2}(\Omega;H^s)$ with $s>3/2$. 
	
\end{enumerate}

To finish the proof of Theorem \ref{Local pathwise solution}  we only need to verify  the blow-up criterion \eqref{blow-up criterion common}.
 Motivated by  \cite{Crisan-Flandoli-Holm-2018-JNS,Rohde-Tang-2020-JDDE}, we first consider, in next lemma, the relationship between the explosion time of $\|u(t)\|_{H^s}$ and the explosion time of $\|u(t)\|_{W^{1,\infty}}$ for \eqref{periodic Cauchy problem}. The results of the lemma will not only 
  immediately imply   the blow-up criterion \eqref{blow-up criterion common}  but also be used  in the next sections.


\begin{Lemma}\label{blow-up criterion lemma}
	Let $(u,\tau^*)$ be the unique maximal solution to \eqref{periodic Cauchy problem}.  Then the real-valued stochastic process $\|u\|_{W^{1,\infty}}$ is also $\mathcal{F}_t$--adapted. Besides, for any $m,n\in\Z^{+}$, define
	\begin{align*}
	\tau_{1,m}=\inf\left\{t\geq0: \|u(t)\|_{H^s}\geq m\right\},\ \ \
	\tau_{2,n}=\inf\left\{t\geq0: \|u(t)\|_{W^{1,\infty}}\geq n\right\}.
	\end{align*}
	For $\displaystyle\tau_1=\lim_{m\rightarrow\infty}\tau_{1,m}$ and $\displaystyle\tau_2=\lim_{n\rightarrow\infty}\tau_{2,n}$,  we have then
	$$
	\tau_{1}=\tau_{2} \ \ \p-a.s.
	$$
\end{Lemma}
\begin{proof}
		To begin with, since $u(\cdot\wedge \tau)\in C([0,\infty);H^s)$ almost surely, we see that for any $t\in[0,\tau]$,
	$$[u(t)]^{-1}(Y)=[u(t)]^{-1}(H^s\bigcap Y),\ \forall\ Y\in\B(W^{1,\infty}).$$ Therefore $u(t)$, as a $W^{1,\infty}$-valued process, is also $\mathcal{F}_t$-adapted. Moreover, the embedding $H^s\hookrightarrow W^{1,\infty}$ for $s>3/2$ means that there is a
	$K=K(s)>0$ such that  $\|\cdot\|_{W^{1,\infty}}<K\|\cdot\|_{H^s}$.
	Then for  every $m\in\N$,
	\begin{align*}
	\sup_{t\in[0,\tau_{1,m}]}\|u(t)\|_{W^{1,\infty}}\leq K\sup_{t\in[0,\tau_{1,m}]}\|u(t)\|_{H^s}
	\leq ([K]+1)m,
\end{align*}
where $[K]$ means the integer part of $K$.
Consequently, 
	$\tau_{1,m}\leq\tau_{2,([K]+1)m}\leq \tau_2$ almost surely,
	which means that
	$\tau_{1}\leq \tau_2$ $\p-a.s.$
	Now we only need to prove the contrary inequality. Let $n,k\in\Z^{+}$, one has
	\begin{align*}
	\left\{\sup_{t\in[0,\tau_{2,n}\wedge k]}\|u(t)\|_{H^s}<\infty\right\}
	=\bigcup_{m\in\Z^{+}}\left\{\sup_{t\in[0,\tau_{2,n}\wedge k]}\|u(t)\|_{H^s}<m\right\}
	\subset\bigcup_{m\in\Z^{+}}\left\{\tau_{2,n}\wedge k\leq\tau_{1,m}\right\}.
	\end{align*}
	Notice that $$\bigcup_{m\in\Z^{+}}\left\{\tau_{2,n}\wedge k\leq\tau_{1,m}\right\}\subset\left\{\tau_{2,n}\wedge k\leq\tau_{1}\right\}.$$
	We see that 
	\begin{align}
	\p\left\{\tau_2\leq\tau_1\right\}
	=\p\left\{\bigcap_{n\in\Z^{+}}\left\{\tau_{2,n}\leq \tau_{1}\right\}\right\}
	=\p\left\{\bigcap_{n,k\in\Z^{+}}\left\{\tau_{2,n}\wedge k\leq \tau_{1}\right\}\right\}=1,\label{tau2<tau1}
	\end{align}
	provided that $\p\left\{\tau_{2,n}\wedge k\leq\tau_{1}\right\}=1$  for all $n,k\in\Z^{+}$.
	To this end, we only need to prove
	\begin{align}
	\p\left\{\sup_{t\in[0,\tau_{2,n}\wedge k]}\|u(t)\|_{H^s}<\infty\right\}=1,\ \
	\forall\ n,k\in\Z^{+}.\label{tau2<tau1 condition}
	\end{align}
	Consider first  $\E\sup_{t\in[0,\tau_{2,n}\wedge k]}\|u(t)\|_{H^s}^2$. We cannot estimate this expectation using the  It\^{o} formula
directly.  Indeed, the It\^{o} formula in a Hilbert space  (\cite[Theorem 4.32]{Prato-Zabczyk-2014-Cambridge} or \cite[Theorem 2.10]{Gawarecki-Mandrekar-2010-Springer})  requires $ \left(\left(u-\g\right)u_x, u\right)_{H^s}$ to be well-defined and the It\^{o} formula  under a  Gelfand triplet (\cite[Theorem I.3.1]{Krylov-Rozovskiui-1979-chapter} or \cite[Theorem 4.2.5]{Prevot-Rockner-2007-book}) requires the dual product ${}_{H^{s-1}}\langle \left(u-\g\right)u_x, u\rangle_{H^{s+1}}$ to be well-defined. In our case we  only have $u\in H^s$ and  $\left(u-\g\right)u_x \in H^{s-1}$ such that neither requirement is fulfilled. 
 Therefore we utilize the  mollifier operator $T_\e$ defined in \eqref{Define Te}. We first apply $T_\e$ to \eqref{periodic Cauchy problem}, and then
use the It\^{o} formula for $\|T_\e u\|_{H^s}^2=\|D^sT_\e u\|_{L^2}^2$ to deduce that for any $n,k>1$ and $t\in[0,\tau_{2,n}\wedge k]$, 
	\begin{align*}
	\|T_\e u(t)\|^2_{H^s}-\|T_\e u(0)\|^2_{H^s}
	=\,\,&2\sum_{k=1}^{\infty}\int_0^{t}
	\left(D^sT_\e h(t',u)e_k,D^sT_\e u\right)_{L^2}{\rm d}W_k\nonumber\\
	&-2\int_0^{t}
	\left(D^s T_\e 
	\left[(u-\g)\partial_xu\right],D^s T_\e u\right)_{L^2}\,{\rm d}t'\nonumber\\
	&-2\int_0^{t}
	\left(D^s T_\e F(u),D^s T_\e u\right)_{L^2} \,{\rm d}t'\nonumber\\
	&+\int_0^{t}
	\sum_{k=1}^\infty\|D^s T_\e h(t',u)e_k\|_{L^2}^2\, {\rm d}t'\nonumber\\
	=:\,\,&\sum_{k=1}^{\infty}\int_0^{t} L_{1,k} \,{\rm d}W_k+\sum_{i=2}^4\int_0^tL_i \, {\rm d}t'.
	\end{align*}
	On account of the Burkholder-Davis-Gundy inequality \eqref{BDG G}, for the expectation of the $H^s$-norm of  $T_\e u$, we arrive  at 
	\begin{align*}
	\E\sup_{t\in[0,\tau_{2,n}\wedge k]}\| T_\e u(t)\|^2_{H^s}\leq\E\| T_\e u_0\|^2_{H^s}
	+C\E\left(\int_0^{\tau_{2,n}\wedge k}\left|\sum_{k=1}^{\infty} L_{1,k}\right|^2{\rm d}t\right)^{\frac12}
	+\sum_{i=2}^4\E\int_0^{\tau_{2,n}\wedge k}|L_i|\,{\rm d}t.
	\end{align*}
	We can infer from \eqref{assumption 1 for h}, \eqref{mollifier property 5}, the stochastic Fubini theorem \cite{Prato-Zabczyk-2014-Cambridge} and Assumption \ref{Assumption-1} that
	\begin{align*}
	\E\left(\int_0^{\tau_{2,n}\wedge k}\left|\sum_{k=1}^{\infty} L_{1,k}\right|^2{\rm d}t\right)^{\frac12}
	\leq& \frac12\E\sup_{t\in[0,\tau_{2,n}\wedge k]}\| T_\e u\|_{H^s}^2
	+Cf^2(n)\int_0^{k}\left(1+\E\|u\|_{H^s}^2\right)\,{\rm d}t.
	\end{align*}
	For $L_2$ and $L_3$, we use Lemma \ref{uux+F Hs inner product} to find
	\begin{align*}
	\E\int_0^{\tau_{2,n}\wedge k}|L_2|+|L_3|\ {\rm d}t\leq Cn\int_0^{k}\left(1+\E\|u\|_{H^s}^2\right){\rm d}t.
	\end{align*}
	Similarly, it follows from the assumption \eqref{assumption 1 for h} that
	\begin{align*}
	\E\int_0^{\tau_{2,n}\wedge k}|L_4|\ {\rm d}t\leq Cf^2(n)\int_0^{k}\left(1+\E\|u\|_{H^s}^2\right){\rm d}t.
	\end{align*}
	If  we combine the above estimates   and use \eqref{mollifier property 5}, we are led for some constant $C=C_n>0$ depending on $n$ to 
	\begin{align*}
	\E\sup_{t\in[0,\tau_{2,n}\wedge k]}\|T_\e u(t)\|^2_{H^s}
	\leq 2\E\|u_0\|^2_{H^s}+ C_n\int_0^{k}\left(1+\E\sup_{t'\in[0,t\wedge\tau_{2,n}]}\|u(t)\|_{H^s}^2\right) \, {\rm d}t.
	\end{align*}
 Since the right hand side of the last  estimate does not depend on $\e$, and
	$T_\e u$ tends to $u$ in $C\left([0,T],H^{s}\right)$ for any $T>0$ almost surely as $\e\rightarrow0$, one can send $\e\rightarrow0$ to obtain
		\begin{align*}
	\E\sup_{t\in[0,\tau_{2,n}\wedge k]}\|u(t)\|^2_{H^s}
	\leq 2\E\|u_0\|^2_{H^s}+ C_n\int_0^{k}\left(1+\E\sup_{t'\in[0,t\wedge\tau_{2,n}]}\|u(t)\|_{H^s}^2\right){\rm d}t.
	\end{align*} 
	Then Gr\"{o}nwall's inequality shows that for each $n,k\in\Z^{+}$, there is a constant $C=C(n,k,u_0)>0$ such that $$\E\sup_{t\in[0,\tau_{2,n}\wedge k]}\|u(t)\|^2_{H^s}<C(n,k,u_0),$$ which gives \eqref{tau2<tau1 condition}.
\end{proof}

We finish the section with the proof of the blow-up criterion in Theorem \eqref{Local pathwise solution}.

\begin{proof}[Proof of \eqref{blow-up criterion common}]
	Let $\tau_{1,m}$, $\tau_{2,n}$, $\tau_1$ and $\tau_2$ be given in Lemma \ref{blow-up criterion lemma}.
If $u$ is the unique pathwise solution with maximal existence time $\tau^*$, for fixed $m,n>0$, even if $\p\{\tau_{1,m}=0\}$ or $\p\{\tau_{2,n}=0\}$ is larger than $0$, for a.e. $\omega\in\Omega$, there is $m>0$ or $n>0$ such that $\tau_{1,m},\tau_{2,n}>0$. By continuity of $\|u(t)\|_{H^s}$ and the uniqueness of $u$, it is easy to check that $\tau_1=\tau_2=\tau^*$. Consequently, we obtain the desired blow-up criterion.
\end{proof} 	

\section{Proof of Theorem \ref{Non breaking}: Strong nonlinear noise}\label{Noise VS blow-up}



To begin with, we note the following algebraic inequality.

\begin{Lemma}\label{log lemma}
	Let $c,M>0$.
	Assume   $$ \text{either }\eta>1,\ a,b>0  \text{ or }
	\eta=1,\ b>a>0.$$
There is  a $C>0$ such that for all $0\leq x\leq M y<\infty$,
	\begin{equation*}
	\frac{a (1+x)y^2+b(1+x)^{\eta}y^2}{1+y^2}-\frac{2b (1+x)^{\eta}y^4}{(1+y^2)^2}
	+\frac{c (1+x)^{\eta}y^4}{(1+y^2)^2(1+\log(1+y^2))}\leq C.
	\end{equation*}
\end{Lemma}
\begin{proof}
Since   $My\geq {x}$,  we have
\begin{align*}
	&\hspace*{-1cm}\frac{a (1+x)y^2+b(1+x)^{\eta}y^2}{1+y^2}-\frac{2b (1+x)^{\eta}y^4}{(1+y^2)^2}
+\frac{c (1+x)^{\eta}y^4}{(1+y^2)^2(1+\log(1+y^2))}\\
\leq \, &a (1+x) +b(1+x)^{\eta}-2b (1+x)^{\eta}\frac{\big(\frac{x}{M})^4}{(1+(\frac{x}{M})^2\big)^2}
+\frac{c (1+x)^{\eta}}{\big(1+\log(1+(\frac{x}{M})^2) \big)}.
\end{align*}
When $\eta>1$ and $a,b>0$ or $\eta=1$ and $b>a>0$, the latter expression tends to $-\infty$ for $x\to \infty$ which implies the statement of the lemma.
\end{proof}

We are now ready to prove Theorem \ref{Non breaking} following  \cite{Ren-Tang-Wang-2020-Arxiv} to large extent.

\begin{proof}[Proof of Theorem \ref{Non breaking}]
	Assume $s>5/2$ and let $u_0$ be an $H^s$-valued $\mathcal{F}_0$-measurable random variable with $\E\|u_0\|^2_{H^s}<\infty$. Let $h(t,u)=h(u)=a \left(1+\|u\|_{W^{1,\infty}}\right)^{\theta}u$ with $\theta>1/2$ and $a\neq0$.
	
	For $r>3/2$, the  embedding $H^r\hookrightarrow W^{1,\infty}$ implies that we have  for any $u,v\in H^r$ the estimate 
$$ \sup_{\|u\|_{H^r},\|v\|_{H^r}\le N}\left\{{\bf 1}_{\{u\ne v\}}  \frac{\|h(u)-h(v)\|_{H^r}}{\|u-v\|_{H^r}}\right\} \le q(N),\ \ N\ge 1. $$
	This means that one can establish the pathwise uniqueness for \eqref{SDGH non blow up Eq} in $H^r$ with $r>3/2$. Hence, in the same way as proving Theorem \ref{Local pathwise solution}, one can show that  \eqref{SDGH non blow up Eq} admits a unique pathwise solution $u$ in $H^s$ with $s>5/2$ 
	and maximal existence time $\tau^*$.  We recall the definition of the mollifier $T_\e$ from  Section \ref{Preliminary Results} and define
	\begin{align*}
	\tau_{m}=\inf\left\{t\geq0: \|u(t)\|_{H^s}\geq m\right\}.
	\end{align*}
	Applying the It\^{o} formula to 
	$\|T_\e u(t)\|^2_{H^s}$ gives
	\begin{align*}
	{\rm d}\|T_\e u\|^2_{H^s}
	= \,\,& 2a\left(1+\|u\|_{W^{1,\infty}}\right)^\theta\left( T_\e u, T_\e u\right)_{H^s}{\rm d}W
	-2\left(T_\e\left[(u-\g)u_x\right], T_\e u\right)_{H^s}{\rm d}t\\
	&-2\left(T_\e F(u), T_\e u\right)_{H^s}{\rm d}t
	+a^2\left(1+\|u\|_{W^{1,\infty}}\right)^{2\theta}\|T_\e u\|_{H^s}^{2}{\rm d}t.\notag
	\end{align*}
	Again, using It\^{o} formula to 
	$\log(1+\|T_\e u\|^2_{H^s})$ yields
	\begin{align*}
	{\rm d}\log(1+\|T_\e u\|^2_{H^s})
	=\,\,&\frac{2a\left(1+\|u\|_{W^{1,\infty}}\right)^\theta}{1+\|T_\e u\|^2_{H^s}}\left(T_\e u, T_\e u\right)_{H^s}{\rm d}W\notag\\
	&-\frac{1}{1+\|T_\e u\|^2_{H^s}}
	\left\{2\left(T_\e \left[(u-\g)u_x\right], T_\e u\right)_{H^s}
	+2\left(T_\e F(u), T_\e u\right)_{H^s}\right\}{\rm d}t\notag\\
	&+\frac{a^2\left(1+\|u\|_{W^{1,\infty}}\right)^{2\theta}}{1+\|T_\e u\|^2_{H^s}}
	\|T_\e u\|_{H^s}^2{\rm d}t
	-2\frac{a^2\left(1+\|u\|_{W^{1,\infty}}\right)^{2\theta}}{(1+\|T_\e u\|^2_{H^s})^2}\|T_\e u\|_{H^s}^{4}{\rm d}t.
	\end{align*}
 Lemma \ref{uux+F Hs inner product} and \eqref{mollifier property 5} imply that there is a $Q=Q(s,c_0,\g)>0$ such that for any $t>0$ we have 
	\begin{align*}
	&\E \log\big(1+\|T_\e u(t\wedge \tau_m)\|^2_{H^s}\big )
	-\E \log\big(1+\|T_\e u_0\|^2_{H^s}\big)\\
	=\,\,&\E\int_0^{t\wedge \tau_m}\frac{1}{1+\|T_\e u\|^2_{H^s}}
	\left\{-2\left(T_\e [(u-\g)u_x], T_\e u\right)_{H^s}
	-2\left(T_\e F(u), T_\e u\right)_{H^s}\right\}{\rm d}t'\notag\\
	&+\E\int_0^{t\wedge \tau_m}
	 \frac{1}{1+\|T_\e u\|^2_{H^s}}
	a^2\left(1+\|u\|_{W^{1,\infty}}\right)^{2\theta}\|T_\e u\|_{H^s}^2\, {\rm d}t'\\
	&-\E\int_0^{t\wedge \tau_m}\frac{2}{(1+\|T_\e u\|^2_{H^s})^2}
	a^2\left(1+\|u\|_{W^{1,\infty}}\right)^{2\theta}
	\|T_\e u\|_{H^s}^{4} \, {\rm d}t'\\
	\leq\,\,&\E\int_0^{t\wedge \tau_m}\left[\frac{1}{1+\|T_\e u\|^2_{H^s}}
	\left\{2Q\left(1+\|u\|_{W^{1,\infty}}\right)\|u\|^2_{H^s}+
	a^2\left(1+\|u\|_{W^{1,\infty}}\right)^{2\theta}\|T_\e u\|_{H^s}^2\right\}\right]\, {\rm d}t'	\\
	&-\E\int_0^{t\wedge \tau_m}\frac{1}{(1+\|T_\e u\|^2_{H^s})^2}
	2a^2\left(1+\|u\|_{W^{1,\infty}}\right)^{2\theta}
	\|T_\e u\|_{H^s}^{4} \, {\rm d}t'.
	\end{align*}
	Notice that for any $T>0$, $T_\e u$ tends to $u$ in $C\left([0,T],H^{s}\right)$ almost surely as $\e\rightarrow0$. Then, by \eqref{mollifier property 5} and the dominated convergence theorem, the last estimate leads to
	\begin{align*}
	&\E \log \big(1+\|u(t\wedge \tau_m)\|^2_{H^s} \big)
	-\E \log(1+\|u_0\|^2_{H^s})\\
	=\,\, &\lim_{\e\rightarrow0}\left(\E \log\big(1+\|T_\e u(t\wedge \tau_m)\|^2_{H^s}\big)
	-\E \log\big(1+\|T_\e u_0\|^2_{H^s}\big)\right)\\
	\leq\,\,&\lim_{\e\rightarrow0}\E\int_0^{t\wedge \tau_m} \frac{1}{1+\|T_\e u\|^2_{H^s}}
	\left\{2Q\left(1+\|u\|_{W^{1,\infty}}\right)\| u\|^2_{H^s}+
	a^2\left(1+\|u\|_{W^{1,\infty}}\right)^{2\theta}\|T_\e u\|_{H^s}^2\right\} \,{\rm d}t'	\\
	&-\lim_{\e\rightarrow0}\E\int_0^{t\wedge \tau_m}\frac{1}{(1+\|T_\e u\|^2_{H^s})^2}
	2a^2\left(1+\|u\|_{W^{1,\infty}}\right)^{2\theta}
	\|T_\e u\|_{H^s}^{4}\,{\rm d}t'\\
	=\,\,&\E\int_0^{t\wedge \tau_m}\frac{2Q\left(1+\|u\|_{W^{1,\infty}}\right)\| u\|^2_{H^s}+
		a^2\left(1+\|u\|_{W^{1,\infty}}\right)^{2\theta}\|u\|_{H^s}^2}{1+\| u\|^2_{H^s}}\,{\rm d}t'\\
	&-\E\int_0^{t\wedge \tau_m}
	-\frac{2a^2\left(1+\|u\|_{W^{1,\infty}}\right)^{2\theta}\| u\|_{H^s}^{4}}{(1+\| u\|^2_{H^s})^2}\,{\rm d}t'.
	\end{align*}
	Since we have assumed \eqref{a theta condition},  Lemma \ref{log lemma} immediately shows that there are constants $K_1,K_2>0$ such that
	\begin{align*}
	&\hspace*{-0.5cm}\E \log\big(1+\| u(t\wedge \tau_m)\|^2_{H^s}\big)-\E \log\big (1+\| u_0\|^2_{H^s}\big)	\\
	\leq\,\,&\E\int_0^{t\wedge \tau_m} K_1-
	K_2\frac{a^2\left(1+\|u\|_{W^{1,\infty}}\right)^{2\theta}\|u\|_{H^s}^4}
	{(1+\|u\|^2_{H^s})^2\left(1+\log(1+\|u\|^2_{H^s})\right)}\,{\rm d}t',
	\end{align*}
	which means that for some $C(u_0,K_1,K_2,t)>0$,
	\begin{align}
	\E\int_0^{t\wedge \tau_m}
	\frac{a^2\left(1+\|u\|_{W^{1,\infty}}\right)^{2\theta}\| u\|_{H^s}^{4}}
	{(1+\|u\|^2_{H^s})^2\left(1+\log(1+\|u\|^2_{H^s})\right)}\,{\rm d}t'
	\leq C(u_0,K_1,K_2,t)<\infty,\label{to use BDG 1}
	\end{align}
	and 
	\begin{align}
	\E\int_0^{t\wedge \tau_m}\left|
	K_1-K_2\frac{a^2\left(1+\|u\|_{W^{1,\infty}}\right)^{2\theta}
		\| u\|_{H^s}^{4}}
	{(1+\|u\|^2_{H^s})^2\left(1+\log(1+\|u\|^2_{H^s})\right)}\right|\,{\rm d}t'
	\leq C(u_0,K_1,K_2,t)<\infty.\label{to use BDG 2}
	\end{align}
	Next, we notice that	there is a function  $\delta:[0,\infty)\to [0,\infty)$  with  $\delta(\e)\rightarrow 0$
	when $\e\rightarrow 0$ such that
	\begin{align*}
	& \hspace*{-0.5cm}\frac{2Q\left(1+\|u\|_{W^{1,\infty}}\right)\|u\|^2_{H^s}
		+	a^2\left(1+\|u\|_{W^{1,\infty}}\right)^{2\theta}\|T_\e u\|_{H^s}^2}{1+\|T_\e u\|^2_{H^s}}
	-\frac{2a^2\left(1+\|u\|_{W^{1,\infty}}\right)^{2\theta}
		\|T_\e u\|_{H^s}^{4}}{(1+\|T_\e u\|^2_{H^s})^2}\\
	\leq\,\,&\frac{2Q\left(1+\|u\|_{W^{1,\infty}}\right)\|u\|^2_{H^s}
		+	a^2\left(1+\|u\|_{W^{1,\infty}}\right)^{2\theta}\|u\|_{H^s}^2}{1+\|u\|^2_{H^s}}
	-\frac{2a^2\left(1+\|u\|_{W^{1,\infty}}\right)^{2\theta}
		\|u\|_{H^s}^{4}}{(1+\|u\|^2_{H^s})^2}+\delta(\e)
	\end{align*}
	holds.
	Therefore,  for any $T>0$, by using Lemma \ref{log lemma}, the Burkholder-Davis-Gundy inequality \eqref{BDG G} and \eqref{to use BDG 2}, we find that
	\begin{align*}
	&\E\sup_{t\in[0,{T\wedge \tau_m}]}\log\big(1+\|T_\e u\|^2_{H^s}\big)-\E \log\big(1+\|T_\e u_0\|^2_{H^s}\big)\\
	\leq\,\,&
	C\E\left(\int_0^{T\wedge \tau_m}
	\frac{a^2\left(1+\|u\|_{W^{1,\infty}}\right)^{2\theta}\|T_\e u\|_{H^s}^4}{\left(1+\|T_\e u\|^2_{H^s}\right)^2}\,
	{\rm d}t\right)^\frac12\\
	&+\E\int_0^{T\wedge \tau_m}
	\left|
	K_1-K_2\frac{a^2\left(1+\|u\|_{W^{1,\infty}}\right)^{2\theta}
		\| u\|_{H^s}^{4}}
	{(1+\|u\|^2_{H^s})^2\left(1+\log(1+\|u\|^2_{H^s})\right)}+\delta(\e)\right|\,{\rm d}t\\
	\leq\,\,&
	\frac12\E\sup_{t\in[0,{T\wedge \tau_m}]}\left(1+\log(1+\|T_\e u\|^2_{H^s})\right)
	+C\E\int_0^{T\wedge \tau_m}
	\frac{a^2\left(1+\|u\|_{W^{1,\infty}}\right)^{2\theta}\|T_\e u\|_{H^s}^4}{\left(1+\|T_\e u\|^2_{H^s}\right)^2
		\left(1+\log(1+\|T_\e u\|^2_{H^s})\right)}\,
	{\rm d}t\\[1.5ex]
	&+K_1T+\E\int_0^{T\wedge \tau_m}K_2\frac{a^2\left(1+\|u\|_{W^{1,\infty}}\right)^{2\theta}
		\| u\|_{H^s}^{4}}
	{(1+\|u\|^2_{H^s})^2\left(1+\log(1+\|u\|^2_{H^s})\right)}\,
	{\rm d}t+\delta(\e)T\\
	\leq\,\,&
	\frac12\E\sup_{t\in[0,{T\wedge \tau_m}]}\left(1+\log(1+\|T_\e u\|^2_{H^s})\right)
	+C\E\int_0^{T\wedge \tau_m}
	\frac{a^2\left(1+\|u\|_{W^{1,\infty}}\right)^{2\theta}\|T_\e u\|_{H^s}^4}{\left(1+\|T_\e u\|^2_{H^s}\right)^2
		\left(1+\log(1+\|T_\e u\|^2_{H^s})\right)}\, 
	{\rm d}t\\[1.5ex]
	&+C(u_0,K_1,K_2,T)+\delta(\e)T.
	\end{align*}
	Thus, we use the dominated convergence theorem, Fatou's lemma  and \eqref{to use BDG 1} to obtain finally 
	\begin{align*}
	&\E\sup_{t\in[0,{T\wedge \tau_m}]}\log\big(1+\| u\|^2_{H^s}\big)
	\leq C(u_0,K_1,K_2,T).
	\end{align*}
	Since $\log(1+x)$ is continuous for $x>0$,
	we have that for any $m\ge1$,
	$$\p\{\tau^*<T\}\leq\p\{\tau_{m}<T\}\leq
	\p\left\{\sup_{t\in[0,T]}\log(1+\|u\|^2_{H^s})\geq \log(1+m^2)\right\}\leq \frac{C(u_0,K_1,K_2,T)}{\log(1+m^2)}.$$
	Letting $m\rightarrow\infty$ forces
	$\p\{\tau^*<T\}=0$ for any $T>0$, which means $\p\{\tau^*=\infty\}=1$.
\end{proof}

\section{Proofs of Theorems \ref{Decay result}-\ref{Blow-up rate}: Non-autonomous linear noise case}\label{global existence and wave breaking}

In this section, we study  \eqref{DGH linear noise 2} with  linear noise. Depending on the strength of the noise  in \eqref{DGH linear noise 2}, we 
provide either  the global existence of pathewise solutions or the  precise  blow-up scenarios for  the maximal pathwise solution. 
 As discussed in Remark \ref{girsanov}, we rely on  the  Girsanov-type transform
\begin{align}
v=\frac{1}{\beta(\omega,t)} u,\ \
\beta(\omega,t)={\rm e}^{\int_0^tb(t') {\rm d} W_{t'}-\int_0^t\frac{b^2(t')}{2} {\rm d}t'}.\label{transform}
\end{align}

We first collect some properties of $v$.

\begin{Proposition}\label{pathwise solutions v}
	Let $s>3/2$, $\a=1$ and $h(t,u)=b(t) u$ such that $b(t)$ satisfies Assumption \ref{Assumption-3}. Let $\s=\left(\Omega, \mathcal{F},\p,\{\mathcal{F}_t\}_{t\geq0}, W\right)$ be fixed in advance. If $u_0(\omega,x)$ is an $H^s$-valued $\mathcal{F}_0$ measurable random variable with $\E\|u_0\|^2_{H^s}<\infty$ and $(u,\tau^*)$ is the corresponding unique maximal solution to \eqref{DGH linear noise 1},
	then for any $c_0,\g\in\R$ and for $t\in[0,\tau^*)$, the process $v$ defined by \eqref{transform} solves the following problem on $\T$ almost surely,
	\begin{equation} \label{periodic Cauchy problem transform}
	\left\{\begin{aligned}
	&v_t+\beta vv_x-\g v_x+\beta (1-\partial_{xx}^2)^{-1}\partial_x\left(v^2+\frac{1}{2} v_x^2\right) +(c_0+\g)(1-\partial_{xx}^2)^{-1}\partial_xv=0,\\
	&v(\omega,0,x)=u_0(\omega,x).
	\end{aligned} \right.
	\end{equation}
	Moreover, we have 
	$v\in C\left([0,\tau^*);H^s\right)\cap C^1([0,\tau^*) ;H^{s-1})$ $\p-a.s.$ and, if $s>3$, then it holds
	\begin{equation}
	\p\big\{\|v(t)\|_{H^1}=\|u_0\|_{H^1}\ \text{for all}\ t\geq0\big\}=1.\label{H1 conservation}
	\end{equation}
\end{Proposition}
\begin{proof}
	Since $b(t)$ satisfies Assumption \ref{Assumption-3},  $h(t,u)=b(t) u$ satisfies Assumption \ref{Assumption-1}. Consequently, Theorem \ref{Local pathwise solution} implies  that \eqref{DGH linear noise 1}  (that is \eqref{periodic Cauchy problem} with $h(t,u)=b(t) u$) has a unique maximal solution $(u,\tau^*)$. 
	A direct computation with the It\^{o} formula yields
	$${\rm d}\frac{1}{\beta}=-b(t)\frac{1}{\beta} {\rm d}W+b^2(t)\frac{1}{\beta} {\rm d}t.$$
	Therefore we arrive at
	\begin{align}
	{\rm d}v
	=\,\,&\frac{1}{\beta}\left[-\left[\left(u-\g\right)\partial_xu+F(u)\right]{\rm d}t+b(t) u \,{\rm d}W\right]+u\left[-b(t)\frac{1}{\beta} {\rm d}W+b^2(t)\frac{1}{\beta} {\rm d}t\right]-b^2(t)\frac{1}{\beta}u\, {\rm d}t\notag\\
	=\,\,&\frac{1}{\beta}\left[-\left(\left(u-\g\right)\partial_xu+F(u)\right){\rm d}t\right]\notag\\
	=\,\,&\left\{-\beta vv_x+\g v_x-\beta (1-\partial_{xx}^2)^{-1}\partial_x\left(v^2+\frac{1}{2} v_x^2\right) -(c_0+\g)(1-\partial_{xx}^2)^{-1}v_x\right\}\,{\rm d}t,\label{v equation}
	\end{align}
	which is \eqref{periodic Cauchy problem transform}$_1$.
	Since $v(0)=u_0(\omega,x)$, we see that $v$  satisfies \eqref{periodic Cauchy problem transform}. Moreover,  Theorem \ref{Local pathwise solution} implies $u\in C\left([0,\tau^*);H^s\right)$ $\p-a.s.$, so is $v$. Besides, from Lemma \ref{F(u) lemma} and \eqref{periodic Cauchy problem transform}$_1$, we see that for a.e. $\omega\in\Omega$, $v_t=\g v_x-\beta vv_x-\beta (F_1(v)+F_2(v))-F_3(v)\in C([0,\tau^*); H^{s-1})$. Hence $v\in C^1\left([0, \tau^*);H^{s-1}\right)$ $\p-a.s.$ 

	Notice that if $s>3$, \eqref{periodic Cauchy problem transform}$_1$ is equivalent to
	\begin{align}
	v_{t}-v_{xxt}+c_0v_x+\g v_{xxx}+3\beta vv_{x}=2\beta v_{x}v_{xx}+\beta vv_{xxx}.\label{random v equation}
	\end{align}
	Multiplying both sides of \eqref{random v equation} by $v$ and then integrating the resulting equation on $x\in\T$, we see that for a.e. $\omega\in\Omega$ and for all $t>0$,
	$$\frac{\rm d}{{\rm d}t}\int_{\T}\left(v^2+v_x^2\right)\, {\rm d}x=0,$$ which implies \eqref{H1 conservation}.
\end{proof}

\subsection{Theorem \ref{Decay result}:  Global existence for weak noise I}

Now we prove the first global existence result, which is  motivated by \cite{GlattHoltz-Vicol-2014-AP,Rockner-Zhu-Zhu-2014-SPTA,Tang-2018-SIMA,Rohde-Tang-2020-JDDE}.

\begin{proof}[Proof of Theorem \ref{Decay result}]
	To begin with, we apply the operator $D^s$ to \eqref{v equation}, multiply both sides of the resulting equation by $D^sv$ and integrate over $\T$ to obtain for   a.e. $\omega\in\Omega$
	\begin{align*}
	\frac{1}{2}\frac{\rm d}{{\rm d}t}\|v(t)\|^2_{H^s}
	=\,\,&\g\int_{\T}D^sv\cdot D^sv_x \,{\rm d}x
	-\beta(\omega,t)	\int_{\T}D^sv\cdot D^s\left[vv_x\right]\,{\rm d}x
	-\beta(\omega,t)\int_{\T}D^sv\cdot D^sF(v)\, {\rm d}x\notag\\
	=\,\,	&-\beta(\omega,t)\int_{\T}D^sv\cdot D^s\left[vv_x\right]\,{\rm d}x
	-\beta(\omega,t)\int_{\T}D^sv\cdot D^sF(v) \,{\rm d}x.
	\end{align*}
	Using Lemma \ref{Kato-Ponce commutator estimate}, integration by parts and Lemma \ref{F(u) lemma},
	we conclude that there is a $C=C(s)>1$ such that for a.e. $\omega\in\Omega$ we have 
	\begin{align*}
	\frac{\rm d}{{\rm d}t}\|v(t)\|^2_{H^s}
	\leq C\beta(t)\ \|v\|_{W^{1,\infty}}\|v\|_{H^s}^2,
	\end{align*}
	where $\beta$ is given in \eqref{transform}	(If necessary, $T_\e$ can be used as in Lemma  \ref{blow-up criterion lemma}).
	Then
	$w={\rm e}^{-\int_0^tb(t') {\rm d} W_{t'}}u={\rm e}^{-\int_0^t\frac{b^2(t')}{2} {\rm d}t'}v$ satisfies
	\begin{align*}
	\frac {\rm d}{ {\rm d}t}\|w(t)\|_{H^s}+\frac{b^2(t)}{2}\|w(t)\|_{H^s}
	\leq C\alpha(\omega,t) \|w(t)\|_{W^{1,\infty}}\|w(t)\|_{H^s},\ \
	\alpha(\omega,t)={\rm e}^{\int_0^tb(t') {\rm d} W_{t'}}. 
	\end{align*}
	Let $R>1$ and $\lambda_1>2$. Assume $\|u_0\|_{H^s}<\frac{b_*}{CK\lambda_1R} <\frac{b_*}{CK\lambda_1}$ almost surely and define
	\begin{align}
	\tau_{1}=\inf\left\{t>0:\alpha(\omega,t) \|w\|_{W^{1,\infty}}
	=\|u\|_{W^{1,\infty}}>\frac{b^2(t)}{C\lambda_1 }\right\}.\label{global time tau}
	\end{align}
 Then it follows from the embedding $\|u(0)\|_{W^{1,\infty}}\leq K\|u(0)\|_{H^s}<\frac{b_*}{C\lambda_1}$ that
	$
	\p\{\tau_{1}>0\}=1,
	$
	and for $t\in[0,\tau_{1})$,
	\begin{align*}
	\frac {\rm d}{ {\rm d}t}\|w(t)\|_{H^s}+\frac{(\lambda_1-2)b^2(t)}{2\lambda_1}\|w(t)\|_{H^s}
	\leq 0.
	\end{align*}
	The above inequality and $w={\rm e}^{-\int_0^tb(t'){\rm d} W_{t'}}u$ imply that for a.e. $\omega\in\Omega$, for any $\lambda_2>\frac{2\lambda_1}{\lambda_1-2}$ and for $t\in[0,\tau_{1})$,
	\begin{align}
	\|u(t)\|_{H^s}
	\leq\,\,& \|w_0\|_{H^s}{\rm e}^{\int_0^tb(t')  {\rm d} W_{t'}-\int_0^t\frac{(\lambda_1-2)b^2(t')}{2\lambda_1}\, {\rm d}t'}\notag\\
	=\,\,&\|u_0\|_{H^s}{\rm e}^{\int_0^tb(t') \,{\rm d} W_{t'}-\int_0^t\frac{b^2(t')}{\lambda_2} {\rm d}t'}{\rm e}^{-\frac{\left((\lambda_1-2)\lambda_2-2\lambda_1\right)}{2\lambda_1\lambda2}\int_0^tb^2(t') \,{\rm d}t'}.
	\label{Extracting some damping}
	\end{align}
	Define the stopping time
	\begin{equation}\label{tau 2 Girsanov}
	\tau_{2}
	=\inf\left\{t>0:{\rm e}^{\int_0^tb(t') {\rm d} W_{t'}-\int_0^t\frac{b^2(t')}{\lambda_2} {\rm d}t'}>R\right\}.
	\end{equation}
	Notice that $\p\{\tau_{2}>0\}=1$. From \eqref{Extracting some damping}, we have that almost surely
	\begin{align}
	\|u(t)\|_{H^s}<\,\,& \frac{b_*}{CK\lambda_1R}\times R\times
	{\rm e}^{-\frac{\left((\lambda_1-2)\lambda_2-2\lambda_1\right)}{2\lambda_1\lambda2}\int_0^tb^2(t') \,{\rm d}t'}\notag\\
	=\,\,&\frac{b_*}{CK\lambda_1}
	{\rm e}^{-\frac{\left((\lambda_1-2)\lambda_2-2\lambda_1\right)}{2\lambda_1\lambda2}\int_0^tb^2(t') \,{\rm d}t'}
	\leq\frac{b_*}{CK\lambda_1},\ \ t\in[0,\tau_{1}\wedge \tau_{2}).
	\label{u energy estimate R theta}
	\end{align}
 By Assumption \ref{Assumption-3}, \eqref{u energy estimate R theta} and \eqref{global time tau}, we find that on $[0,\tau_{1}\wedge \tau_{2})$,
	$$\|u(t)\|_{W^{1,\infty}}\leq K\|u(t)\|_{H^s}\leq\frac{b_*}{C\lambda_1}\leq \frac{b^2(t)}{C\lambda_1 }\ \ \p-a.s.,$$
which means
	\begin{equation}\label{tau 1 > tau 2}
	\p\{\tau_{1}\geq\tau_{2}\}=1.
	\end{equation}
	Therefore it follows from \eqref{u energy estimate R theta} that
	$$\p\left\{
	\|u(t)\|_{H^s}<\frac{b_*}{CK\lambda_1}
	{\rm e}^{-\frac{\left((\lambda_1-2)\lambda_2-2\lambda_1\right)}{2\lambda_1\lambda2}\int_0^tb^2(t') \,{\rm d}t'}	\ {\rm\ for\ all}\ t>0
	\right\}\geq
	\p\{\tau_{2}=\infty\}.$$
	We apply \ref{exit time eta} in Lemma \ref{eta Lemma} to find that
	\begin{equation*}
	\p\{\tau_{2}=\infty\}>1-\left(\frac{1}{R}\right)^{2/\lambda_2},
	\end{equation*}
	which completes the proof.
\end{proof}

\subsection{Theorem \ref{Global existence result}:  Global existence for weak noise  II}
Let $\beta(\omega,t)$ be given as in \eqref{transform}. With Proposition \eqref{pathwise solutions v} at hand, we can proceed to prove Theorem \ref{Global existence result}.
We see that for a.e. $\omega\in\Omega$, the transform $v(\omega,t,x)$  solves \eqref{periodic Cauchy problem transform} on $[0,\tau^*)$. Moreover, since $H^s\hookrightarrow C^2$ for $s>3$, we have $v,v_x\in C^1\left([0, \tau^*)\times\T\right)$. Then for a.e. $\omega\in\Omega$, for any $x\in\T$ and $c_0,\g\in\R$, the problem
\begin{equation} \label{particle line}
\left\{\begin{aligned}
&\frac{{\rm d}q(\omega,t,x)}{{\rm d}t}=\beta(\omega,t)v(\omega,t,q(\omega,t,x))-\g,\ \ \ \ t\in[0,\tau^*),\\
&q(\omega,0,x)=x,\ \ \ x\in \T,
\end{aligned} \right.
\end{equation}
has a unique solution $q(\omega,t,x)$ such that $q(\omega,t,x)\in C^1([0,\tau^*)\times \T)$ for a.e $\omega\in\Omega$.
Moreover, differentiating \eqref{particle line} with respect to $x$ yields that for a.e. $\omega\in\Omega$,
\begin{equation*}
\left\{\begin{aligned}
&\frac{dq_x(\omega,t,x)}{dt}=\beta(\omega,t)v_x(\omega,t,q)q_x,\ \ \ \ t\in[0,\tau^*),\\
&q_x(\omega,0,x)=1,\ \ \ x\in \T.
\end{aligned} \right.
\end{equation*}
For a.e. $\omega\in\Omega$, we solve the above equation to obtain
$$q_x(\omega,t,x)=\exp{\left(\int_0^t\beta(\omega,t')v_x(\omega,t',q(\omega,t',x))\ {\rm d}t'\right)}.$$ Thus for a.e. $\omega\in\Omega$, $q_x>0$, $(t,x)\in[0, \tau^*)\times \T$.  
On the other hand, if $v$ solves \eqref{periodic Cauchy problem transform} (or equivalently \eqref{random v equation}) $\p-a.s.$, then the momentum variable $V=v-v_{xx}$ satisfies
\begin{equation}\label{V equation}
V_{t}+c_0 v_x+\beta vV_{x}+2\beta V v_{x}+\gamma v_{xxx}=0\ \ \p-a.s.
\end{equation}
Particularly, if $c_0+\g=0$, \eqref{V equation} becomes
\begin{equation*} 
V_{t}+c_0 v_x+\beta vV_{x}+2\beta V v_{x}+\gamma v_{xxx}=V_{t}-\g V_x+\beta vV_{x}+2\beta V v_{x}=0\ \ \p-a.s.,
\end{equation*}
which means
\begin{align*}
\frac{\rm d}{{\rm d}t}\left[V(\omega,t,q(\omega,t,x))q_x^{2}(\omega,t,x)\right]
=&q_x^2\left[V_t+\beta v V_x-\g V_x+2\beta V v_x\right]=0\ \ \p-a.s.
\end{align*}
This, and $q_x(\omega,0,x)=1$ imply that $$V(\omega,t,q(\omega,t,x))q_x^{2}(\omega,t,x)=V_0(\omega,x).$$ Consequently, we have ${\rm sign}(V)={\rm sign}(V_0)$. Besides, since $v=G_{\T}*V$ with $G_{\T}>0$ given in \eqref{Helmboltz operator}, we have ${\rm sign}(v)={\rm sign}(V)$. Summarizing the above analysis, we have the following result:
\begin{Lemma}\label{same sign with initial data}
	Assume $c_0+\g=0$ and $s>3$. Let $V_0(\omega,x)=(1-\partial_{xx}^2)u_0(\omega,x)$ and $V(\omega,t,x)=v(\omega,t,x)-v_{xx}(\omega,t,x)$, where $v(\omega,t,x)$ solves \eqref{periodic Cauchy problem transform} on $[0,\tau^*)$ $\p-a.s.$ Then for a.e. $\omega\in\Omega$,  
	\begin{align*}
	{\rm sign}(v)={\rm sign}(V)={\rm sign}&(V_0),\ \ (t,x)\in[0, \tau^*)\times \T.
	\end{align*}
\end{Lemma}

The next step is to  control $\|u(\omega,t)\|_{W^{1,\infty}}$.  In combination 
  with \eqref{blow-up criterion common}, we will then  directly verify Theorem \ref{Global existence result}.

\begin{Lemma}\label{vx bounded lemma}
	Let all the conditions as in the statement of Proposition \ref{pathwise solutions v} hold true. Let $V$ and $V_0$ be defined in Lemma \ref{same sign with initial data}. If additionally we have $c_0+\g=0$ and
	\begin{align*}
	\p\{V_0(\omega,x)>0,\  \forall\ x\in\T\}=p,\ \
	\p\{V_0(\omega,x)<0,\  \forall\ x\in\T\}=q,
	\end{align*}
	for some $p,q\in[0,1]$, then the maximal solution $(u,\tau^\ast)$ of \eqref{DGH linear noise 2} satisfies
	\begin{equation*}
	\p\Big\{\|u_x(\omega,t)\|_{L^{\infty}}\leq \|u(\omega,t)\|_{L^{\infty}}\lesssim\beta(\omega,t)\|u_0\|_{H^1},\  \forall\ t\in[0,\tau^*)\Big\}\geq p+q.
	\end{equation*}
\end{Lemma}
\begin{proof}
	Using \eqref{Helmboltz operator}, one can derive (see \cite{Tang-2018-SIMA}) that
 for a.e. $\omega\in \Omega$,
and
for all $(t,x)\in[0, \tau^*)\times \T$,
\begin{align}
\left[v+v_x\right](\omega,t,x)=\,\,&\frac{1}{2\sinh(\pi)}
\int^{2\pi}_0{\rm e}^{(x-y-2\pi\left[\frac{x-y}{2\pi}\right]-\pi)}V(\omega,t,y)\ {\rm d}y,\label{u+ux}\\
\left[v-v_x\right](\omega,t,x)
=\,\,&\frac{1}{2\sinh(\pi)}
\int^{2\pi}_0{\rm e}^{(y-x+2\pi\left[\frac{x-y}{2\pi}\right]+\pi)}V(\omega,t,y)\ {\rm d}y.\label{u-ux}
\end{align}
	Then one can employ \eqref{u+ux}, \eqref{u-ux} and Lemma \ref{same sign with initial data} to obtain that for a.e. $\omega\in \Omega$ and for all $(t,x)\in[0, \tau^*)\times \T$,
	\begin{equation}\label{ux u set}
	\left\{\begin{aligned}
	-v(\omega,t,x)\leq v_x(\omega,t,x)\leq v(\omega,t,x),\ \ & {\rm if}\ \ V_0(\omega,x)=(1-\partial_{xx}^2)u_0(\omega,x)>0, \\
	v(\omega,t,x)\leq v_x(\omega,t,x)\leq -v(\omega,t,x),\ \ & {\rm if}\ \ V_0(\omega,x)=(1-\partial_{xx}^2)u_0(\omega,x)<0.
	\end{aligned} \right.
	\end{equation}
	Notice that
	\begin{align}\label{non intersection}
	\left\{V_0(\omega,x)>0\right\}\cap \left\{V_0(\omega,x)<0\right\}=\emptyset.
	\end{align}
	Combining \eqref{ux u set} and \eqref{non intersection} yields
	\begin{equation}
	\p\Big\{|v_x(\omega,t,x)|\leq |v(\omega,t,x)|,\  \forall\ (t,x)\in[0, \tau^*)\times \T\Big\}\geq p+q.\label{v<vx probability}
	\end{equation}
	In view of $H^1\hookrightarrow L^{\infty}$, \eqref{H1 conservation} and \eqref{v<vx probability}, we arrive at
	\begin{equation*}
	\p\Big\{\|v_x(\omega,t)\|_{L^{\infty}}
	\leq \|v(\omega,t)\|_{L^{\infty}}
	\lesssim \|v(\omega,t)\|_{H^1}
	=\|u_0\|_{H^1},\ \forall\ t\in[0,\tau^*)\Big\}\geq p+q.
	\end{equation*}
	Via \eqref{transform}, we obtain the desired estimate.
\end{proof}

\begin{proof}[Proof of Theorem \ref{Global existence result}]
	
	Let $(u,\tau^*)$ be the maximal solution to \eqref{DGH linear noise 2}. 
	Then Lemma \ref{vx bounded lemma} implies that 
	\begin{equation*}
	\p\Big\{\|u\|_{W^{1,\infty}}\lesssim 2\beta(\omega,t)\|u_0\|_{H^1},\  \forall\ t\in[0,\tau^*)\Big\}\geq p+q.
	\end{equation*}
	It follows from \ref{eta->0} in Lemma \ref{eta Lemma} that $\sup_{t>0}\beta(\omega,t)<\infty$ $\p-a.s.$ Then we can infer from \eqref{blow-up criterion common} that $\p\{\tau^*=\infty\}\geq p+q$.
	That is to say,
	$
	\p\left\{u\ {\rm exists\ globally}\right\}\geq p+q
	$.
\end{proof}

\subsection{Theorem \ref{blow-up criterion weak noise}: Blow-up scenario}

\begin{proof}[Proof of Theorem \ref{blow-up criterion weak noise}]
		Recall \eqref{transform}. By \ref{eta->0} in Lemma \ref{eta Lemma},
	$A=A(\omega)=\sup_{t>0}\beta(\omega,t)<\infty$ $\p-a.s.$ Then we can first infer from
	$H^1\hookrightarrow L^{\infty}$ and \eqref{H1 conservation} that for all $t>0$,
	\begin{align*}
	\sup_{t>0}\|u\|_{L^{\infty}}\lesssim A\|u_0\|_{H^{1}}<\infty\ \ \p-a.s.,
	\end{align*}
	which is \eqref{blow-up by WB 1}. Now we prove \eqref{blow-up by WB 2}.
	Let $$\Omega_1=\left\{\limsup_{t\rightarrow \tau^*}\|u(t)\|_{H^{s}}=\infty\right\}\ \text{and}\ \  \Omega_2=\left\{\liminf_{t\rightarrow \tau^*}\left[\min_{x\in\T}u_x(t,x)\right]=-\infty\right\}.$$
	By the previously proven blow-up criterion in Theorem \ref{Local pathwise solution}, we have that for a.e. $\omega\in\Omega_2$, $ \omega\in\Omega_1 $. Now we prove that for a.e. $\omega\in\Omega_1$, $ \omega\in\Omega_2 $.
	Suppose not. Then there is a positive random variable $K=K(\omega)<\infty$ almost surely such that	$$u_x(\omega,t,x)>-K,\ \ (t,x)\in[0,\tau^*(\omega))\times\T\ \ \p-a.s.$$  
Using \eqref{V equation}, \eqref{transform} and integration by parts we find that
	\begin{align*}
	\frac{{\rm d}}{{\rm d}t}\int_{\T}V^{2}\ {\rm d}x &
	=2\int_{\T}V [-\beta vV_{x}-2\beta V v_{x}+\gamma V_x]\ {\rm d}x\\
	&=-4\beta \int_{\T}V^{2}v_x\ {\rm d}x-2\beta\int_{\T} V V_xv\ {\rm d}x\\
	&=-3\beta\int_{\T}V^{2}v_x\ {\rm d}x\leq 3K \int_{\T}V^{2}\ {\rm d}x,\ \  t\in [0,\tau^*)\ \ \p-a.s.,
	\end{align*}
	which yields that 
	\begin{equation*}
	\|V\|_{L^2}\lesssim {\rm e}^{3Kt}\|V(0)\|_{L^2}<\infty,\ \ t\in [0,\tau^*)\ \ \ \p-a.s..
	\end{equation*}
	Combining the above estimate, \eqref{transform} and  $A(\omega)=\sup_{t>0}\beta(\omega,t)<\infty$ $\p-a.s.$ (cf. \ref{eta->0} in Lemma \ref{eta Lemma}), we have that
	\begin{equation*} 
	\|u(t)\|_{H^2}\lesssim \beta(t){\rm e}^{K t}\|u(0)\|_{H^2}<\infty,\ \ t\in [0,\tau^*)\ \ \p-a.s.
	\end{equation*}
	By the embedding $H^2\hookrightarrow W^{1,\infty}$ and the blow-up criterion in Theorem \ref{Local pathwise solution}, almost surely we have that $\|u(t)\|_{H^{s}}$ can be extended beyond $\tau^*$. Therefore we obtain a contradiction and hence $\omega\in\Omega_2$. Therefore we obtain \eqref{blow-up by WB 2}.
\end{proof}

\subsection{Theorem \ref{Wave breaking}: Wave breaking and its probability}

 The proof of Theorem \ref{Wave breaking} relies on certain properties of the solution $v$ to the problem \eqref{periodic Cauchy problem transform}.
\begin{Proposition}\label{blow-up sDGH Pro}
	Let $\s=(\Omega, \mathcal{F},\p,\{\mathcal{F}_t\}_{t\geq0}, W)$ be a fixed stochastic basis, let $b(t)$ satisfy Assumption \ref{Assumption-3}, $c_0+\g=0$, $s>3$ and $u_0=u_0(x)\in H^s$ be an $H^s$-valued $\mathcal{F}_0$ measurable random variable with $\E\|u_0\|^2_{H^s}<\infty$. Let $(u,\tau^*)$ be the maximal solution  to \eqref{DGH linear noise 1} with initial random variable $u_0$. 
	Recall the process $\beta$ given in \eqref{transform} and the constant $\lambda$ as in  equation \eqref{f2<f H1}. Let $N=\frac{\lambda}{2}\|u_0\|_{H^1}^2<\infty$. Then for $v$, defined by \eqref{transform}, we have that
	\begin{equation}\label{M=min vx}
	M(\omega,t):=\min_{x\in\T}[v_x(\omega,t,x)]
	\end{equation} 
	satisfies the following estimate almost surely:
	\begin{align}
	\frac{{\rm d}}{{\rm d}t}M(t)
	\leq& \beta N-\beta \frac{1}{2}M^2(t)\ \  {\rm  a.e.\ on}\ (0,\tau^*).\label{blow-up DGH dM 1}
	\end{align}
 Moreover, 	
	if $\displaystyle M(0)<-\sqrt{2N}$ almost surely, then 
	\begin{align}
	M(t)\leq -\sqrt{2N},\ \forall\ t\in[0,\tau^*) \ \ \p-a.s.,\label{M<0}
	\end{align}
	with $M$ being nonincreasing on $[0,\tau^*)$ $\p-a.s.$
\end{Proposition}
\begin{proof}
	For any $v\in H^1$, it is easy to see that, cf. \cite{Constantin-2000-JNS},
	\begin{align}
	G_{\T}*\left(v^2+\frac{1}{2}v^2_x\right)(x)\geq \frac{1}{2}v^2.\label{blow-up DGH G*>}
	\end{align}
	Using \eqref{F decomposition}, \eqref{Helmboltz operator}, \eqref{v equation} and \eqref{transform}, we find for  $c_0+\g=0$ that
	\begin{equation}\label{blow-up equation DGH}
	v_{tx}-\g v_{xx}+\beta vv_{xx}
	=\beta v^2-\beta \frac{1}{2}v^2_x
	-\beta G_{\T}*\left(v^2+\frac{1}{2}v^2_x\right),\ \ t\in[0,\tau^*)\ \ \ \p-a.s.
	\end{equation}
	By Proposition \ref{pathwise solutions v}, $v(\omega,t,x)\in C^1([0, \tau^*);H^{s-1})$ with $s>3$ almost surely. To apply Lemma \ref{Constantin-Escher} for each path, we recall \eqref{M=min vx} and let $z(\omega, t)$ be a point where the infimum of $v_x$ is attained as in Lemma \ref{Constantin-Escher}. Then for a.e. $\omega\in\Omega$, $v_{xx}(t,z(\omega,t))=0$. Moreover, Lemma \ref{Constantin-Escher} also implies that  for a.e. $\omega\in\Omega$, the path of $M(\omega,t)$ is locally Lipschitz.
	Then for almost all $t\in[0,\tau^*)$, evaluating \eqref{blow-up equation DGH} in $(t,z(t))$ with using Lemma \ref{Constantin-Escher} yields for a.e.  $\omega\in\Omega$,
	\begin{equation}\label{dM equation}
	\frac{{\rm d}}{{\rm d}t}M(t)=\beta v^2(t,z(t))-\beta \frac{1}{2}M^2(t)-\beta G_{\T}*\left(v^2+\frac{1}{2}v^2_x\right)(t,z(t))\ \ {\rm a.e.\ on}\ (0,\tau^*).
	\end{equation}
	Since $\E\|u_0\|_{H^s}<\infty$, $N=\frac{\lambda}{2}\|u_0\|_{H^1}^2<\infty$ $\p-a.s.$ 
	Applying \eqref{blow-up DGH G*>},  \eqref{f2<f H1} and
	\eqref{H1 conservation} in the above equation gives that for a.e. $\omega\in\Omega$,
	\begin{align*}
	\frac{{\rm d}}{{\rm d}t}M(t)\leq\,\,& \beta \frac{1}{2}v^2(t,z(t))-\beta \frac{1}{2}M^2(t)\notag\\
	\leq\,\,& \beta \frac{\lambda}{2}\|v(t)\|_{H^1}^2-\beta \frac{1}{2}M^2(t)\notag\\
	=\,\,& \beta N-\beta \frac{1}{2}M^2(t)\ \  {\rm  a.e.\ on}\ (0,\tau^*),
	\end{align*}
	which is \eqref{blow-up DGH dM 1}. In order to show \eqref{M<0}, 
we define $\tau$ as
	\begin{equation*} 
	\tau(\omega):=
	\inf\left\{t>0: M(\omega,t)>-\sqrt{2N}\right\}\wedge \tau^*.
	\end{equation*}
	If $M(0)<-\sqrt{2N}$, then $\p\{\tau>0\}=1$.
	Now we only need to  show that 
	\begin{align}
	\p\{\tau(\omega)=\tau^*(\omega)\}=1.\label{dM<0 2}
	\end{align}
	Actually,
	failure of \eqref{dM<0 2} would ensure the existence of a set  $\Omega'\subseteq\Omega$ such that $\p\{\Omega'\}>0$ and $0<\tau(\omega')<\tau^*(\omega')$ for a.e. $\omega'\in\Omega'$. In view of the time continuity of $M$ (recall Lemma \ref{Constantin-Escher}), we find that
	$
	M(\omega',\tau(\omega'))=-\sqrt{2N}.
	$
	From \eqref{blow-up DGH dM 1} we have that $M(\omega',t)$ is nonincreasing for $t\in[0,\tau(\omega'))$. Hence by the continuity of the path of $M(\omega',t)$ again, we see that $M(\omega',\tau(\omega'))\leq M(0)<-\sqrt{2N}$,
	which is a contradiction. Hence \eqref{dM<0 2} is true and so is \eqref{M<0}. 
\end{proof}

\begin{Proposition}\label{blow-up sDGH Pro 2}
	Let all the conditions as in Proposition \ref{blow-up sDGH Pro} hold true.
	Let $0<c<1$ and 
	\begin{equation*}
 \Omega^*=\left\{\omega:\beta(t)\geq c{\rm e}^{-\frac{b^*}{2}t }\ \text{for\ all}\ t\right\}.
	\end{equation*}
	If $\displaystyle M(0)<-\frac{1}{2}\sqrt{\frac{(b^*)^2}{c^2}+8N}-\frac{b^*}{2c}$ almost surely,  then for a.e. $\omega\in\Omega^*$,
	\begin{align*}
	\tau^*(\omega)<\infty.
	\end{align*}
\end{Proposition}
\begin{proof}
	We rewrite \eqref{blow-up DGH dM 1} as
	\begin{align*}
	\frac{{\rm d}}{{\rm d}t}M(t)\leq - \frac{\beta}{2}\left(1-\frac{2N}{M^2(0)}\right)M^2(t)
	-\frac{\beta N}{M^2(0)}M^2(t)+\beta N
	\ \ 	\ {\rm a.e.\ on}\ (0,\tau^*)\ \ \ \p-a.s.
	\end{align*}
	Due to Proposition \ref{blow-up sDGH Pro}, we have  
	\begin{align*}
	\frac{{\rm d}}{{\rm d}t}M(t)\leq& - \frac{\beta(t)}{2}\left(1-\frac{2N}{M^2(0)}\right)M^2(t)
	-\left(\frac{M^2(t)}{M^2(0)}-1\right)\beta(t) N\notag\\
	\leq& - \frac{\beta(t)}{2}\left(1-\frac{2N}{M^2(0)}\right)M^2(t)
	\ \ 	\ {\rm a.e.\ on}\ (0,\tau^*)\ \ \ \p-a.s.
	\end{align*}
	Since $M(t)$  is locally Lipschitz continuous in $t$ and satisfies \eqref{M<0}, $\frac{1}{M(t)}$ is also locally Lipschitz continuous in $t$ almost surely. Therefore an integration  leads to
	\begin{align*}
	\frac{1}{M(t)}-\frac{1}{M(0)}
	\geq\left(1-\frac{2N}{M^2(0)}\right)\int_{0}^{t}\frac{\beta(t')}{2}\ {\rm d}t',\ \ t\in(0,\tau^*)\ \ \p-a.s.,
	\end{align*}
	which together with \eqref{M<0} means that for a.e. $\omega\in\Omega^*$,
	\begin{align*}
	-\frac{1}{M(0)}
	\geq\left(\frac{1}{2}-\frac{N}{M^2(0)}\right)\int_{0}^{\tau^*}\beta(t)	\ {\rm d}t
	\geq \left(\frac{1}{2}-\frac{N}{M^2(0)}\right)\left(\frac{2c}{b^*}-\frac{2c}{b^*}{\rm e}^{-\frac{b^*}{2}\tau^*}\right).
	\end{align*}
	Recall that $\displaystyle M(0)<-\frac{1}{2}\sqrt{\frac{(b^*)^2}{c^2}+8N}-\frac{b^*}{2c}$ almost surely. We finally arrive at
	\begin{align*}
	\left(\frac{1}{2}-\frac{N}{M^2(0)}\right)\frac{2c}{b^*}{\rm e}^{-\frac{b^*}{2}\tau^*}
	\geq\frac{2c}{b^*}\left(\frac{1}{2}-\frac{N}{M^2(0)}\right)+\frac{1}{M(0)}>0\ \ \text{a.e.\ on}\ \Omega^*.
	\end{align*}
	Therefore we have $\tau^*<\infty$ a.e. on $\Omega^*$.
\end{proof}

\begin{proof}[Proof of Theorem \ref{Wave breaking}] 	Proposition \ref{blow-up sDGH Pro 2}  implies that 
	$$\p\left\{
	\tau^*<\infty
	\right\}\geq\p\left\{\beta(t)\geq c{\rm e}^{-\frac{b^*}{2}t }\ \text{for\ all}\ t\right\}.$$
	Since $b^2(t)<b^*$ for all $t>0$, we have
	$$\left\{{\rm e}^{\int_0^t b(t'){\rm d}W_{t'}}>c\ \text{for\ all}\ t\right\}
	\subseteq
	\left\{\beta(t)\geq c{\rm e}^{-\frac{b^*}{2}t }\ \text{for\ all}\ t\right\}.$$
	Therefore we arrive at
	$$\p\left\{
	\tau^*<\infty
	\right\}\geq\p\left\{{\rm e}^{\int_0^t b(t'){\rm d}W_{t'}}>c\ \text{for\ all}\ t\right\}>0,$$
	which gives the desired estimate in 	Theorem \ref{Wave breaking}.
\end{proof}

\subsection{Theorem \ref{Blow-up rate}: Wave breaking rate} 
As the last contribution of the paper,  we prove Theorem \ref{Blow-up rate}, which provides a precise bound on the wave breaking rate.

\begin{proof}[Proof of Theorem \ref{Blow-up rate}.]
	Recalling \eqref{dM equation}, we have that almost surely
	\begin{align*}
-\beta\left\|G_{\T}*\left(v^2+\frac{1}{2}v^2_x\right)\right\|_{L^\infty}\leq\frac{{\rm d}}{{\rm d}t}M(t)+\beta \frac{1}{2}M^2(t)\leq\beta \|v\|^2_{L^\infty}
	\ \ 	\ {\rm a.e.\ on}\ (0,\tau^*).
	\end{align*}
	Using  $\|G\|_{L^\infty}<\infty$ and \eqref{H1 conservation}, we have $$\left\|G_{\T}*\left(v^2+\frac{1}{2}v^2_x\right)\right\|_{L^\infty}\lesssim\left\|v^2+\frac{1}{2}v^2_x\right\|_{L^1}\lesssim \|v\|^2_{H^1}=\|u_0\|^2_{H^1}.$$
	Therefore there is a constant $C>0$ such that 
	\begin{align}\label{M estimate blow-up rate}
	-C\beta\|u_0\|^2_{H^1}\leq\frac{{\rm d}}{{\rm d}t}M(t)+\beta \frac{1}{2}M^2(t)\leq C\beta\|u_0\|^2_{H^1}
	\ \ 	\ {\rm a.e.\ on}\ (0,\tau^*).
	\end{align}
	Let $\varepsilon \in (0,\frac{1}{2})$ and $K=C\|u_0\|^2_{H^1}$. Since $\liminf_{t\rightarrow \tau^*}M(t)=-\infty$ a.e. on $\{\tau^*<\infty\}$ and $K<\infty$ almost surely (cf. Theorem \ref{blow-up criterion weak noise}), for a.e. $\omega\in\Omega$,
	there is some $t_{0}=t_0(\omega,\e) \in (0,\tau^*)$ with $M(t_{0})<0$ and $M^{2}(t_{0})>\frac{K}{\varepsilon}$. Similar to the proof of \eqref{M<0}, we have that for a.e. $\omega\in\{\tau^*<\infty\}$,
	\begin{equation}\label{M^2 lower bound}
	M^{2}(t)>\frac{K}{\varepsilon},\ \ t\in [t_{0},\tau^*).
	\end{equation}
	A combination of \eqref{M estimate blow-up rate} and \eqref{M^2 lower bound} enables us to infer that for a.e. $\omega\in\{\tau^*<\infty\}$,
	\begin{equation*}
	\beta\frac{K}{M^2(t)}+ \frac{\beta}{2} >-\frac{\frac{{\rm d}}{{\rm d}t}M(t)}{M^2(t)}>-\beta \frac{K}{M^2(t)}+  \frac{\beta}{2}\  \ \ {\rm a.e.\ on } \ (t_0,\tau^*).
	\end{equation*}
	Since for a.e. $\omega\in\{\tau^*<\infty\}$, the path of $M$ is locally Lipschitz in $t$ and satisfies \eqref{M^2 lower bound}, $\frac{1}{M}$ is also locally Lipschitz. 
	Then we integrate the above estimate on $(t,\tau^*)$ to derive that for a.e. $\omega\in\{\tau^*<\infty\}$,
	\begin{equation*}
	\left(\frac{1}{2}+\varepsilon\right) \int_{t}^{\tau^*}\beta(t') \,{\rm d}t'\geq -\frac{1}{M(t)}\geq \left(\frac{1}{2}-\varepsilon\right) \int_{t}^{\tau^*}\beta(t')\,{\rm d}t',\ \ t_0<t<\tau^*.
	\end{equation*}
	Therefore we can infer from \eqref{transform} and \eqref{M=min vx} that for a.e. $\omega\in\{\tau^*<\infty\}$,
	\begin{equation*}
	\frac{1}{\frac{1}{2}+\varepsilon} \leq -\min_{x\in\T}[u_x(\omega,t,x)]\beta^{-1}(t)\int_{t}^{\tau^*}\beta(t')\,{\rm d}t'\leq \frac{1}{\frac{1}{2}-\varepsilon} ,\ \ t_0<t<\tau^*.
	\end{equation*}
	Since $\varepsilon \in (0,\frac{1}{2})$ is arbitrary, we obtain that for $\beta(\omega,t)={\rm e}^{\int_0^tb(t')\, {\rm d} W_{t'}-\int_0^t\frac{b^2(t')}{2} \,{\rm d}t'}$,
	\begin{equation*}
	\lim_{t\rightarrow \tau^*} \left( \min_{x\in\T}[u_x(t,x)]\int_{t}^{\tau^*}\beta(t')\,{\rm d}t'\right) =-2\beta(\tau^*) \ \  \text{a.e.\ on}\ \ \{\tau^*<\infty\},
	\end{equation*}
	which completes the proof.
\end{proof}

\section*{Acknowledgement}
The authors would like to express their great gratitude to the anonymous referees for their 
important suggestions, which have led to 
a significant improvement of this paper. Hao Tang 
 is deeply indebted to Professor Feng-Yu Wang for his insightful suggestions on the Lyapunov-type condition on the noise.


\end{document}